\documentclass[12pt,reqno]{amsart}
\usepackage{amsfonts,amsthm,amsmath,amssymb,amscd}
\usepackage{mathabx} 
\usepackage{esint}
\usepackage{graphics}
\usepackage{indentfirst}
\usepackage{latexsym}
\usepackage[dvips]{epsfig}
\usepackage{mathrsfs}
\usepackage{color}
\usepackage{enumitem}
\usepackage{hyperref}
\usepackage{color, soul}
\date{}
\allowdisplaybreaks \allowdisplaybreaks[2]

\numberwithin{equation}{section}

\usepackage{float}
\usepackage{graphics}
\usepackage{epsfig}
\usepackage{tikz}
\usepackage{amssymb}

\newcommand{\norm}[1]{\left\lVert #1 \right\rVert}
\newcommand{\abs}[1]{\left\lvert #1 \right\rvert}
\newcommand{\subeqref}[2]{$ \eqref{#1}_{#2} $}

\def\p{\partial}
\def\R{\mathbb{R}}

\def\T{\mathcal{T}}

\def\ah{\hat{a}}
\def\th{\hat{T}}
\def\wh{\hat{w}}
\def\at{\tilde{a}}
\def\wt{\tilde{w}}
\def\supt{\mathop{sup}\limits_{t \in [0,\mathcal{T}]}}
\def\ph{\mathbf{\Phi}}
\def\w{\mathbf{W}}
\def\vsx{v^S_{\mathcal{X}}}

\def\avl{\overline{v}_l}
\def\avr{\overline{v}_r}
\def\aul{\overline{u}_l}
\def\aur{\overline{u}_r}
\def\vt{\tilde{v}}

\def\ut{\tilde{u}}

\def\X{\mathcal{X}}
\def\Y{\mathcal{Y}}

\newtheorem{theorem}{Theorem}[section]

\newtheorem{lemma}[theorem]{Lemma}
\newtheorem{proposition}[theorem]{Proposition}
\newtheorem{remark}[theorem]{Remark}

\newenvironment{pf}{{\noindent \it \bf Proof:}}{{\hfill$\Box$}\\}

\begin{document}

\title[viscoelasticity with non-convexity under periodic perturbations]{Stability of Large-Amplitude Viscous Shock Under Periodic Perturbation for 1-d Viscoelasticity with Non-Convex Constitutive Relations}
	
	\author{Yu~Mei}
	\address[Yu Mei]{School of Mathematics and Statistics, Northwestern Polytechnical University, Xi'an 710129, China}
	\email{\href{mailto:yu.mei@nwpu.edu.cn}{yu.mei@nwpu.edu.cn}}
	
	\author{Peng Yuan*}
	\thanks{*Corresponding Author}
	\address[Peng Yuan]{School of Mathematics and Statistics, Northwestern Polytechnical University, Xi'an 710129, China}
	\email{\href{mailto:pengyuanq@foxmail.com}{pengyuanq@foxmail.com}}

% REQUIRED
\keywords{system of viscoelasticity, viscous shock waves,  nonlinear stability,  non-convexity, periodic perturbation, large-amplitude}

% REQUIRED
\subjclass[2020]{35B40, 35L65, 35L67, 76A10}

\maketitle
\begin{abstract}
	This paper investigates the large-time behavior of the viscous shock profile for the one-dimensional system of viscoelasticity, subject to initial perturbations that approach space-periodic functions at far fields. We specifically address the case with non-convex constitutive stress relations and non-degenerate Lax's shock. Under the assumptions of suitably small initial perturbations satisfying a zero-mass type condition, we prove that the solution of the system converges to a viscous shock profile with a shift, which is partially determined by the space-periodic perturbation. Notably, our result imposes no amplitude restrictions on the viscous shock waves.  This work extends the result of Kawashima-Matsumura (\textit{Commun. Pure Appl. Math.} \textbf{47}, 1994) by simultaneously handling both large-amplitude shocks and space-periodic perturbations, while also generalizing the result of Huang-Yuan (\textit{Commun. Math. Phys.} \textbf{387}, 2021) by allowing for a non-convex constitutive relation. The key ingredient of proof is decomposing the large-amplitude shock wave into small-amplitude shocks and, for each, introducing suitable transform and weight functions to counteract the adverse effects of non-convex constitutive relations encountered during weighted energy estimates on the system in effective velocity and deformation gradient variables.
\end{abstract}

%\tableofcontents

%%%%%%%%%%%%%%%%%%%%%%%%%%%%%%%%%%%%%%%%%%%%%%%%%%%%%
%
%%%%%%%%%%%%%%%%%%%%%%%%%%%%%%%%%%%%%%%%%%%%%%%%%%%%%
\section{Introduction} 
In this paper, we consider the following one-dimensional system of viscoelasticity with non-convex constitutive relations  
\begin{equation}\label{VNC-Ion}
	\begin{cases}
		\p_tv-\p_xu=0,&\quad\\
		\p_tu-\p_x{\sigma(v)}=\mu\p^2_xu,&\quad\\
	\end{cases}x\in \R, t>0,
\end{equation}
where $v$ and $u$ denote the deformation gradient and the velocity, $\mu>0$ is the viscous constant, and $\sigma(v)$ represents the stress, which is assumed to be a $C^3(\R)$ function and satisfy the following sign conditions:
\begin{equation}\label{SC}
	\begin{aligned}
		\sigma'(v)>0,\quad \text{for all}~~v ~~\text{under consideration,}
	\end{aligned}
\end{equation}
\begin{equation}\label{SC2}
	\begin{aligned}
		\sigma''(v)\gtrless 0,\quad \text{for} ~~ v\gtrless0 ~~\text{under consideration,}
	\end{aligned}
\end{equation}
\begin{equation}\label{SC3}
	\sigma'''(v)>0, \quad \text{for}~ v\neq0~\text{under consideration,}
\end{equation}
so that $\sigma(v)$ has a point of inflection at $v=0$. The system \eqref{VNC-Ion} with $\mu=0$ is strictly hyperbolic with the two different characteristic roots
\begin{equation}\label{characteristic roots}
	\lambda=\pm \lambda(v),\qquad \text{where}~\lambda(v)=\sqrt{\sigma'(v)}.
\end{equation}
Moreover, the characteristics are neither genuinely nonlinear nor linearly degenerate. We are concern about  the Cauchy problem of \eqref{VNC-Ion} with the initial data
\begin{equation}\label{initial}
	(v,u)(x,0)=(v_0(x),u_0(x)),~~x\in\R,
\end{equation} 
satisfying
\begin{equation}\label{cond-farfield}
	(v_0,u_0)(x)~\to 
	\begin{cases}
		(\overline{v}_l,\overline{u}_l)+(\phi_{0l},\psi_{0l}) \qquad \text{as}~x~\rightarrow -\infty\\
		(\overline{v}_r,\overline{u}_r)+(\phi_{0r},\psi_{0r}) \qquad \text{as}~x~\rightarrow +\infty
	\end{cases}
\end{equation}	
where $(\overline{v}_l,\overline{u}_l)$, $(\overline{v}_r,\overline{u}_r)$ are given constants, $(\phi_{0l},\psi_{0l}) \in \R^2$ and $(\phi_{0r},\psi_{0r}) \in \R^2$ are  $\pi_l$-, $\pi_r$-periodic functions respectively. 
Without loss of generality, we assume that the constant states $(\overline{v}_{l,r},\overline{u}_{l,r})$ generate a single 2-shock to the Riemann problem of the system \eqref{VNC-Ion} with $\mu=0$
and initial data
\begin{equation}\label{eini}
	(v,u)(x,0)=(v_0,u_0)(x)=
	\begin{cases}
		(\avl,\aul),x<0,\\
		(\avr,\aur),x>0.
	\end{cases}
\end{equation}
That is,  $(\overline{v}_{l,r},\overline{u}_{l,r})$ satisfy the Rankine-Hugoniot condition,
\begin{equation}\label{RH-cond}
	\begin{cases}
		-s(\overline{v}_r-\overline{v}_l)-(\overline{u}_r-\overline{u}_l)=0,\\
		-s(\overline{u}_r-\overline{u}_l)-(\sigma(\overline{v}_r)-\sigma(\overline{v}_l))=0
	\end{cases}
\end{equation}
and the generalized shock condition
\begin{equation}\label{degenerate}
 \lambda(\overline{v}_r)\leq s<\lambda(\overline{v}_l),
\end{equation}
where $\lambda(v)$ is given in \eqref{characteristic roots} and $s$ is given by
\begin{equation}\label{shock speed}
	s=\sqrt{\frac{\sigma(\overline{v}_r)-\sigma(\overline{v}_l)}{\overline{v}_r-\overline{v}_l}}.
\end{equation}
From now on,  we only consider the non-degenerate shock, namely 
\begin{equation}\label{non-degenerate}
	\sigma'(\overline{v}_r)< s^2<\sigma'(\overline{v}_l),
\end{equation} 
and assume that $\overline{v}_r>\overline{v}_l, \avl\avr<0$ without loss of generality.
We also assume the periodic perturbations in \eqref{cond-farfield}  have zero averages
\begin{equation}\label{0-A}
	\int_{0}^{\pi_l}(\phi_{0l},\psi_{0l})\,dx=0 \quad and \quad \int_{0}^{\pi_r}(\phi_{0r},\psi_{0r})\,dx=0.
\end{equation}
The system \eqref{VNC-Ion} admits a viscous shock profile $(v^S,u^S)(\xi)=(v^S,u^S)(x-st)$ connecting $(\overline{v}_l,\overline{u}_l)$ and $(\overline{v}_r,\overline{u}_r)$ satisfying \eqref{RH-cond} and \eqref{non-degenerate}(c.f.\cite{InitialModel}). That is, there exists a traveling wave solution $(v^S,u^S)(\xi)$ to \eqref{VNC-Ion} satisfying
\begin{equation}\label{profile}
	\begin{cases}
		-s(v^S)'(\xi)-(u^S)'(\xi)=0,\qquad \xi \in \R,\\
		-s(u^S)'(\xi)-(\sigma(v^S))'(\xi)=\mu (u^S)'',\quad \xi \in \R\\
		(v^S,u^S)(\xi) \to (\overline{v}_l,\overline{u}_l)(resp.(\overline{v}_r,\overline{u}_r)) \quad \text{as}~\xi \to -\infty(resp.+\infty).
	\end{cases}
\end{equation}

The stability of viscous shock waves for system of viscous conservation laws, for examples compressible Navier-Stokes equations, etc., is an important issue from both mathematical and physical viewpoints. 
When the periodic perturbations vanish, (i.e. $\phi_{0l,0r}=\psi_{0l,0r}=0$ in \eqref{cond-farfield}),  Il'in and Ole\v inik in their pioneering work \cite{Oleinik.1960} studied the scalar conservation law with strictly convex flux, and proved the stability of shocks and rarefaction waves by using the maximum principle. For compressible Navier-Stokes equations and general parabolic system with genuinely nonlinear characteristics, Matsumura and Nishihara \cite{zero-mass},
and Goodman \cite{Goodman.1986} independently developed $L^2$-energy method to prove the stability of viscous shock, provided that the shock-amplitude is small and the initial perturbation is small and of zero mass. Later, Kawashima and Matsumura \cite{MatsumuraCMP} also obtained the stability of shocks for several systems in gas dynamics.
The zero-mass condition was further removed in Liu \cite{Liu1}, Szepessy and Xin \cite{Xin1}, and  Liu \cite{Liu2}  by introducing diffusion
waves to carry excessive masses and establishing the point-wise estimates. For large-amplitude shocks, Zumbrun and his collaborators, in a series of work \cite{Zumbrun1998,Mascia2004, Zumbrun2009,Zumbrun2017,Zumbrun2016}, showed the spectral stability under small initial perturbations for various viscous conservation laws including compressible Navier-Stokes, MHD equations and so on.
Vasseur and Yao \cite{Yao.2016} took advantage of
``effective velocity" to prove the nonlinear stability of the large-amplitude but  small initial perturbations viscous shock for the isentropic Navier-Stokes
equations, which was further extended by He and Huang \cite{HehuangJDE} to the case with not strictly convex pressure. Concerning with large initial perturbations but small-amplitude shocks, Kang and Vasseur recently established
the $L^2$-contraction properties for viscous conservation laws in \cite{vasseur2,vasseurjmes}, and proved the stability of shock in \cite{vasseurinvent} for compressible Navier-Stokes equations allowing vanishing viscosity limits. Together with Wang, they also obtained the stability of composite wave of a viscous shock and a rarefaction wave in \cite{vasseurwang}.  See also \cite{GoodmanTAMS,wangjems} and the references for the stability of planar shock waves. For the viscous conservation laws with non-convex flux, Kawashima and Matsumura \cite{InitialModel} proved the stability of viscous shock for the scalar case as well as the viscoelasticity system \eqref{VNC-Ion}-\eqref{SC3} by weighted energy method, provided that the shock-amplitude is small and the initial perturbation is small and of zero mass. Furthermore, the stability results for the degenerate shock and more general flux case were obtained in  \cite{Matsumura-Nishihara1995,Nishihara.1995,meiM3AS,Meiming.1997}.  Later on, for the viscoelasticity system \eqref{VNC-Ion} with a non-convex stress of the opposite sign to \eqref{SC2},  Matsumura and Mei \cite{Gaijin} established the stability of viscous shock with small initial perturbations and shock speed $s$ satisfying some conditions (condition (2.18) in \cite{Gaijin}) by introducing a tranform functions to get a new equivalent system and using the weighted energy method. Very recently, Huang, Wang and Zhang\cite{HuangWangZhang} obtained the stability of composite waves of degenerate Oleinik shock and rarefaction for cubic non-convex scalar conservation laws by a new type of $a$-contraction method. We also refer to \cite{LiuGuo} for the stability of shock in the case of not 
strictly convex stress, and \cite{Liu1988, kawashima1993} for the viscoelasticity model with fading memory.

To the best of our knowledge, until now, there has been no result on the stability of large-amplitude shock for the system of viscoelasticity with non-convex constitutive relations. The aim of this paper is to investigate the large-time stability of the viscous shock $(v^S,u^S)$ for the Cauchy problem  \eqref{VNC-Ion},\eqref{initial}, with non-convex constitutive relations \eqref{SC}-\eqref{SC3} in the case of non-degenerate shock \eqref{non-degenerate} and periodic perturbations of constant far fields \eqref{cond-farfield}. This non-localized perturbation of Riemann solutions to the hyperbolic system was firstly studied by Xin, Yuan, and Yuan \cite{Xin.2019,Xin.2021,Yuan.2020}, where they proved the stability of shocks and rarefaction waves with periodic perturbations at far fields for the 1-d scalar conservation laws in both inviscid and viscous cases. Recently, Huang and Yuan \cite{HuangYuanCMP} established the stability of large-amplitude viscous shock under
periodic perturbation for $1$-$d$ compressible isentropic Navier–Stokes
equations. We also refer to \cite{MeiYuan2025} for the stability of shock and rarefaction waves of Navier-Stokes-Poisson system and \cite{YuanSiam,YuanSCM2025} for the planar shocks of viscous scalar conservation law and compressible Navier-Stokes equations. It should be remarked that Lax \cite{Lax1} and
Glimm-Lax \cite{Lax2} were the first to study the periodic solutions to hyperbolic conservation laws. They showed that for the scalar equations and $2\times2$ systems, the periodic solutions time-asymptotically decay to their constant averages. When we study the stability of shocks under periodic perturbations, some new difficulties arise, for example, the solution to the viscous system carries the periodic oscillations at infinity in space so the difference between the solution and the viscous shock profile is not integrable on $\R$. In \cite{Xin.2019,Xin.2021,Yuan.2020,HuangYuanCMP,MeiYuan2025}, they constructed a suitable ansatz by selecting two appropriate shift curves $\X(t)$ and $\Y(t)$. The ansatz carries the same oscillations of the solution $(v,u)$
at the far fields, so that the difference $(v-\tilde v,u-\tilde u)$ belongs to the $H^2(\R)$ space and thus the energy method based on anti-derivative variables is available. 

 The main result can be roughly stated as follows: if the initial perturbation around the viscous shock is suitably small and satisfies a zero-mass type condition, then the solution $(v,u)$ of the Cauchy problem \eqref{VNC-Ion},\eqref{initial}, with non-convex constitutive relations \eqref{SC}-\eqref{SC3} and non-degenerate shock condition  \eqref{non-degenerate}, tends to the viscous shock profile $(v^S,u^S)(\xi)$ with a shift as time goes to the infinity, where the shift is partially determined by the space-periodic oscillations. The rigorous theorem will be stated in Theorem \ref{mainresult}. It should be remarked that our result imposes no amplitude restrictions on the shocks and extend the result in \cite{InitialModel,Gaijin} by removing the small amplitude assumption of shock and allowing space-periodic perturbations, while also generalize the result in \cite{HuangYuanCMP} to the non-convex constitutive relation case.

Now let us introduce the strategies in the proof. As in many previous works, the general strategy consists of the two main steps:
\begin{itemize}
	\item [1.] Construct a suitable ansatz $(\tilde{v},\tilde{u})$ with the desired asymptotic behavior, and derive the system of ansatz with some decay residual terms.
	\item[2.] Take the difference between the original system and ansatz system $(v-\tilde{v}, u-\tilde{u})$,  and show the global solutions in terms of anti-derivative variables with large time behavior  $(v-\tilde{v}, u-\tilde{u})\rightarrow 0$ as $t\rightarrow+\infty$ for small perturbed initial data.
\end{itemize}
In step 1, we can construct the ansatz $(\tilde{v},\tilde{u})$, by the same way as in Huang and Yuan \cite{HuangYuanCMP} for compressible Navier-Stokes equations with convex pressure, to carry the same oscillations of the solution $(v,u)$ at the far fields, and try to use the effective velocity $m:=u-\mu v$ to deal with the large amplitude. However, in step 2, when the constitutive relation is non-convex like \eqref{SC2},  the energy estimates of the anti-derivative variable of the deformation gradient and the effective velocity will produce a bad term involving $\sigma''(v)$ with  a change sign near $v=0$, which can not be controlled.  This difficulty was overcame by Kawashima and Matsumura \cite{InitialModel} for the small amplitude shock without periodic perturbation case by the weighted energy method and Matsumura and Mei \cite{Gaijin} for the shock satisfying a stronger condition than Lax's shock \eqref{non-degenerate} by  the weighted energy method applied on a new system after a transformation.  Motivated by \cite{Gaijin} and \cite{InitialModel}, 
we also apply the method of transformed and  weighted energy estimates to the equations of deformation gradient $v$ and the effective velocity $m$. However, in order to remove the constraint on shock speed $s$ in \cite{Gaijin}, it is subtle to construct the transform function and weight function. In fact, the condition in \cite{Gaijin} requires some smallness of shock amplitude. Our idea is that we decompose the large-amplitude shock satisfying \eqref{non-degenerate} into small-amplitude shocks, and 
design two groups of transform functions and weight functions for each small-amplitude shocks so that the method of transformed and  weighted energy estimates is available. As a cost, our weighted and transform functions are no longer $C^1$, but only piecewise smooth functions. Therefore, 
our construction of weight and transform functions have to ensure that the jump of flux  in the process of 
energy estimate should be non-positive on each non-smooth point. In section 3.1.2 below, we construct the weighted and transform functions separately in  $[\avl,0]$ and $[0,\avr]$, and extend them a little bit to $[\avl,v_d]$ and $[v_{-1},\avr]$, and divide the energy estimates on $\mathbb{R}$ into two parts with
$\vsx$ in $[\avl,0]$ and $[0,\avr]$. By performing the weighted energy method to the transformed system \eqref{transystem} in terms of anti-derivative variables of $(v,m)$ for viscous shock $\vsx$ in $[\avl,0]$ and $[0,\avr]$ respectively, and adding them together, we can obtain the basic energy estimate. Then, we go back to the system  \eqref{system2} to achieve the desired estimates of higher-order derivatives.

This paper is organized as follows. In Sect.\ref{section2}, we give some properties of viscous shock wave and periodic solutions, and then construct the ansatz and state the main result. In Sect.\ref{section3}, we derive the a priori estimates  through the anti-derivative method for the perturbations, and the main theorem is proved in Sect.\ref{section4}. Some tedious calculations about the flux are supplemented in the appendices. For the notational conventions, we denote $\norm\cdot =\norm\cdot_{L^2(\R)}$ and $\norm \cdot _k=\norm \cdot_{H^k(\R)}$ for $k \geqq 1$ throughout this paper.

\section{Ansatz and Main Result}\label{section2}

\subsection{Properties of viscous shock and periodic solutions}

We first recall some properties of the viscous shock profile in \eqref{profile} whose proof refer to  Section 3.1 in \cite{InitialModel}.
\begin{lemma}\label{unique}
	(\cite{InitialModel})~Suppose that~\eqref{SC}~and~\eqref{SC2}~hold. Let $(\overline{v}_l,\overline{u}_l)\neq (\overline{v}_r,\overline{u}_r)$ satisfy \eqref{RH-cond} and \eqref{degenerate}. Then
	\begin{enumerate}
		\item[(i)] there exists a shock profile $(v^S,u^S)(x-st)$ of \eqref{VNC-Ion}
		which connects  $(\overline{v}_l,\overline{u}_l)$ and $(\overline{v}_r,\overline{u}_r)$. The $(v^S,u^S)(\xi)$ is unique up to a shift $\xi$, and both
		$v^S(\xi)$ and $u^S(\xi)$ are monotone functions of $\xi$.
		\item[(ii)] For $s>0$ in \eqref{shock speed}. When $\overline{v}_r>\overline{v}_l$ (and hence $\overline{u}_r <  \overline{u}_l$ by \eqref{RH-cond}), we have
		\begin{equation}\label{mono}
			\begin{aligned}
			&\overline{v}_r > v^S(\xi) > \overline{v}_l,\qquad (v^S(\xi))'> 0,\\
			&\overline{u}_r <  u^S(\xi) <  \overline{u}_l,\qquad (u^S(\xi))'<  0,\\
			&s\mu(v^S)'=-s^2v^S+\sigma(v^S)-a\equiv h(v^S),
			\end{aligned}
		\end{equation}
		where
		\begin{equation}
			a=-s^2\overline{v}_{l,r}+\sigma(\overline{v}_{l,r})
		\end{equation}
	so that $h((\overline{v}_{l,r})=0$.
	\end{enumerate}
\end{lemma}
\begin{lemma}\label{4cond}
	(\cite{InitialModel})~Suppose that \eqref{SC} and \eqref{SC2} hold. Then there four conditions are equivalent to each other.
	\begin{enumerate}
		\item [(i)]When $\overline{v}_r > \overline{v}_l$, $h(v) > 0$ for $v \in (\overline{v}_l,\overline{v}_r).$
		\item [(ii)]$\sigma'(\overline{v}_r)\leq s^2,$ i.e., $\lambda(\overline{v}_r)\leq s$.
		\item [(iii)]$\sigma'(\overline{v}_r)\leq s^2<\sigma'(\overline{v}_l)$, i.e., $\lambda(\overline{v}_r)\leq s<\lambda(\overline{v}_l)$
		\item [(iv)]There exists uniquely a $v^* \in (\overline{v}_l,\overline{v}_r)$ with the following properties: $\sigma'(v^*)=s^2$ and 
		\begin{equation}\label{sbm}
			s^2>\sigma'(v) \quad \text{for} \quad v \in (v^*,\overline{v}_r) \qquad s^2<\sigma'(v) \quad \text{for} \quad v \in (\overline{v}_l,v^*)
		\end{equation}
	\end{enumerate}
		Moreover, if one of the above four conditions holds, then we must have $\overline{v}_l \neq 0$. In addition, $\overline{v}_r > \overline{v}_l$
and $v^*< 0$ hold when $\overline{v}_l< 0$.
\end{lemma}

The next lemma shows that periodic solutions to \eqref{VNC-Ion} have exponential decay. 
\begin{lemma}\label{decay}
	Assume that $(v_0,u_0)(x) \in H^k(0,\pi)$ with $k\geqq 2$ is periodic with period $\pi > 0$ and average $(\overline{v},\overline{u})$. Then there
	exists $\varepsilon_0>0$ such that if 
	\begin{align*}
		\varepsilon:=\norm{(v_0,u_0)-(\overline{v},\overline{u})}_{H^k(0,\pi)}\leq \varepsilon_0,
	\end{align*}
	the problem \eqref{VNC-Ion} with initial data $(v_0,u_0)$ admits a unique periodic solution
	\begin{align*}
		(v,u)(x,t)\in C(0,+\infty;H^k(0,\pi)),
	\end{align*}
	which has some period and average as $(v_0,u_0)$. Moreover, it holds that 
	\begin{equation}\label{decay1}
		\norm{(v,u)-(\overline{v},\overline{u})}_{H^k(0,\pi)}(t)\leq C\varepsilon e^{-\alpha t},~t\geqq 0,
	\end{equation}
	where the constant $C>0$ and $\alpha>0$ are independent of $\varepsilon$ and t.
\end{lemma}
The proof of Lemma \ref{decay} is similar to the one for compressible Navier-Stokes system studied in \cite{HuangYuanCMP}. We omit it. Here in below, we denote $(v_{l,r},u_{l,r})(x,t)$ to be the unique periodic solution to \eqref{VNC-Ion} with the periodic initial data 
\begin{align*}
	(v_{l,r},u_{l,r})(x,0)=(\overline{v}_{l,r},\overline{u}_{l,r})+(\phi_{0l,0r},\psi_{0l,0r})(x),
\end{align*}
where $(\phi_{0l,0r},\psi_{0l,0r})$ is given in \eqref{cond-farfield}. 

\subsection{Ansatz and its equations}

Now we construct the ansatz, and write their equations. For a viscous shock profile $(v^S,u^S)(x-st)$ solving \eqref{profile}, let
\begin{equation}\label{gfuntion}
	g(x):=\frac{v^S(x)-\avl}{\avr-\avl}=\frac{u^S(x)-\aul}{\aur-\aul},
\end{equation}
where the equality can follow easily from \eqref{RH-cond} and \ref{profile}. It is straightforward to check from Lemma \ref{mono} that
$0<g(x)<1$ and $g'(x)>0$ for all $x \in \R$.

For any function $f(x)$ and the shift curve $\xi(t)$, we denote $f_\xi$ to be the shifted function
\begin{align*}
	f_\xi(x):=f(x-\xi(t)) \quad \text{with the derivatives}~f^{(k)}_\xi(x):=f^{(k)}(x-\xi(t)),~k\geqq 1.
\end{align*} 
As in \cite{HuangYuanCMP}, we construct the ansatz as follows,
	\begin{equation}\label{ansatz}
		\begin{aligned}
			\vt(x,t):=v_l(x,t)(1-g_{st+\X(t)}(x))+v_r(x,t)g_{st+\X(t)}(x),\\
			\ut(x,t):=u_l(x,t)(1-g_{st+\Y(t)}(x))+u_r(x,t)g_{st+\Y(t)}(x),
		\end{aligned}
	\end{equation}
where $\X(t)$ and $\Y(t)$ are two $C^1$ curves on $[0,+\infty)$ to be determined. Note that $(\vt,\ut)$ approaches the periodic solution 
$(v_l,u_l)$ (resp.$(v_r,u_r)$) as $x\to -\infty$ (resp.$+\infty$), thus we expect the ansatz $(\vt,\ut)$ carries same oscillations with the
solution $(v,u)$ at far fields, i.e. $(v-\vt,u-\ut)(x,t) \to 0$ as $|x| \to +\infty$ for all $t\geqq 0$.

By plugging the ansatz $(\vt,\ut)$ into \eqref{VNC-Ion}, we have the following system
\begin{equation}\label{VNC-T}
	\begin{cases}
		\p_t \vt-\p_x \ut=-\p_x F_1-f_2,\\
		\p_t\ut-\p_x\sigma(\vt)-\mu\p^2_x \ut=-\p_xF_3-f_4,
	\end{cases}
\end{equation}
where
\begin{equation}\label{F1234}
	\begin{aligned}
		&F_1=(u_r-u_l)(g_{st+\Y}-g_{st+\X}),\\
		&f_2=(u_r-u_l)g'_{st+\X}+(v_r-v_l)g'_{st+\X}(s+\X'),\\
		&F_3=(\sigma(\vt)-\sigma(v_l))(1-g_{st+\Y})+(\sigma(\vt)-\sigma(v_r))g_{st+\Y}+\mu(u_r-u_l)g'_{st+\Y},\\
		&f_4=(u_r-u_l)g'_{st+\Y}(s+\Y')+(\mu\p_xu_r-\mu\p_xu_l)g'_{st+\Y}+(\sigma(v_r)-\sigma(v_l))g'_{st+\Y}.
	\end{aligned}
\end{equation}
It should be remarked that both $F_1(x,t)$ and $F_3(x,t)$ vanish as $|x| \to +\infty$ for any $t\geqq 0$ since $\vt-v_l=(v_r-v_l)g_{st+\X}$ and $\vt-v_r=-(v_r-v_l)(1-g_{st+\X})$.

In order to use the anti-derivative method, we need to find appropriate shift curves $\X(t)$ and $\Y(t)$ so that  
\begin{equation}\label{zero-mass}
	\int_{\R}(v-\vt,u-\ut)(x,t)  \,dx=0 \quad \text{for all}~t\geqq 0. 
\end{equation}
Note that \eqref{VNC-Ion} and \eqref{VNC-T} yield  
\begin{equation}\label{zero-mass2}
	\frac{d}{dt}\int_{\R}(v-\vt,u-\ut) (x,t) \,dx=\int_{\R}(f_2,f_4)(x,t) \,dx. 
\end{equation}
Thus, \eqref{zero-mass} holds if
\begin{equation}\label{zero-mass3}
	\int_{\R} (f_2,f_4)(x,t) \,dx=0, \quad t>0, 
\end{equation}
and
\begin{equation}\label{zero-mass4}
	\int_{\R}(v_0(x)-\vt(x,0),u_0(x)-\ut(x,0))  \,dx =0.
\end{equation}
It follows from \eqref{zero-mass3} that $\X$ and $\Y$ satisfy
\begin{equation}\label{XY}
	\begin{cases}
		\begin{aligned}
			&\X'(t)=-s-\frac{\int_{\R}(u_r-u_l)g'_{st+\X}(x)  \,dx }{\int_{\R}(v_r-v_l)g'_{st+\X}(x)  \,dx },\\
			&\Y'(t)=-s-\frac{\int_{\R}[\sigma(v_r)-\sigma(v_l)+\mu\p_x u_r-\mu\p_x u_l]g'_{st+\Y}(x)\,dx }{\int_{\R}(u_r-u_l)g'_{st+\Y}(x) \,dx}.
		\end{aligned}
	\end{cases}
\end{equation}
While \eqref{zero-mass4} yields that $\X(0)=\X_0$ and $\Y(0)=\Y_0$ satisfy
\begin{equation}\label{XY1}
	\begin{cases}
		\begin{aligned}
				&\mathcal{A}_1(\X_0)+ \frac{1}{\avr-\avl}\left\{\int_{-\infty}^{0}(v_0-v^S-\phi_{0l})  \,dx +\int_{0}^{+\infty}(v_0-v^S-\phi_{0r})  \,dx \right\}=0,\\
				&	\mathcal{A}_2(\Y_0)+ \frac{1}{\aur-\aul}\left\{\int_{-\infty}^{0}(u_0-u^S-\psi_{0l})  \,dx +\int_{0}^{+\infty}(u_0-u^S-\psi_{0r})  \,dx \right\}=0,
		\end{aligned}
	\end{cases}
\end{equation}
where
\begin{equation}\label{XY2}
	\begin{cases}
		\begin{aligned}
			\mathcal{A} _1(\X_0):=\X_0+\frac{1}{\avr-\avl}\left\{\int_{-\infty}^{0}(\phi_{0l}-\phi_{0r})g_{\X_0}  \,dx-\int_{0}^{+\infty}(\phi_{0l}-\phi_{0r})(1-g_{\X_0})  \,dx\right\},\\
			\mathcal{A} _2(\Y_0):=\Y_0+\frac{1}{\aur-\aul}\left\{\int_{-\infty}^{0}(\psi_{0l}-\psi_{0r})g_{\Y_0}  \,dx-\int_{0}^{+\infty}(\psi_{0l}-\psi_{0r})(1-g_{\Y_0})  \,dx\right\}.
		\end{aligned}
	\end{cases}
\end{equation}
The existence and exponentially decay to limiting location properties of shifts $\X(t),\Y(t)$ are given by the following lemma.
\begin{lemma}\label{XYdecay}
	Assume that the periodic perturbations $\phi_{0l},\phi_{0r},\psi_{0l}$ and $\psi_{0r}$ satisfy \eqref{0-A}. Then there exists $\varepsilon_0>0$ such that
	if 
	\begin{align}\label{small assumption for shifts}
		\varepsilon:=\norm{\phi_{0l},\psi_{0l}}_{H^2(0,\pi_l)}+\norm{\phi_{0r},\psi_{0r}}_{H^2(0,\pi_r)}\leq \varepsilon_0,
	\end{align}
	there exists a unique solution $(\X,\Y)(t) \in C^1[0,+\infty)$ to \eqref{XY} with the initial data $(\X_0,\Y_0)$. Moreover, the solution satisfies that
	\begin{center}
		\begin{minipage}[c]{7cm}
			$ |\X'(t)|+\left|\mathcal{X}(t)-\mathcal{X}_{\infty}\right| \leq  C \varepsilon e^{-\alpha t},$ \\
		  $|\Y'(t)|+\left|\mathcal{Y}(t)-\mathcal{Y}_{\infty}\right| \leq C \varepsilon e^{-\alpha t},$
		\end{minipage}
		\begin{minipage}[c]{2cm} 
			$ t \geqq 0,$
		\end{minipage}
	\end{center}
	where the constants $C>0$ and $\alpha>0$ are independent of $\varepsilon$ and t, and the shifts $\X_{\infty}$ and $\Y_{\infty}$ are given by $\X_{\infty}=\mathcal{A}_1(\X_0)+\mathcal{C} _1$
	and $\Y_{\infty}=\mathcal{A} _2(\Y_0)+\mathcal{C} _2$, with
	\begin{equation}\label{C1}
		\mathcal{C} _1:=\frac{1}{\avr-\avl}\left\{\frac{1}{\pi_l}\int_{0}^{\pi_l} \int_{0}^{x} \phi_{0l} \,dy  \,dx -\frac{1}{\pi_r}\int_{0}^{\pi_r} \int_{0}^{x} \phi_{0r} \,dy  \,dx\right\},
	\end{equation}
	\begin{equation}\label{C2}
		\begin{aligned}
		&\mathcal{C}_2:=\frac{1}{\aur-\aul}\left\{\frac{1}{\pi_l}\int_{0}^{\pi_l} \int_{0}^{x} \psi_{0l}(y) \,dy  \,dx+\int_{0}^{+\infty}\frac{1}{\pi_l}\int_{0}^{\pi_l}[\sigma(v_l(x,t))-\sigma(\avl)] \,dx \,dt\right.\\ 
		&\left.-\frac{1}{\pi_r}\int_{0}^{\pi_r} \int_{0}^{x} \psi_{0r}(y) \,dy  \,dx-\int_{0}^{+\infty}\frac{1}{\pi_r}\int_{0}^{\pi_r}[\sigma(v_r(x,t))-\sigma(\avr)] \,dx \,dt\right\}
	    \end{aligned}
	\end{equation}
\end{lemma}
\begin{pf}
	The proof of Lemma follows from the argument as in \cite{HuangYuanCMP} for compressible Navier-Stoke equations. We only sketch the idea and refer the readers to \cite{HuangYuanCMP} for details. 
	
	Firstly, since \eqref{XY2} yields that 
	\begin{align*}
		|\mathcal{A}_1'(\X_0)-1 |\leq \frac{\norm{\phi_{0l},\phi_{0r}}_{L^{\infty}(\R)}}{|\avr-\avl|}\norm{g'}_{L^1(\R)}=\frac{\norm{\phi_{0l},\phi_{0r}}_{L^{\infty}(\R)}}{|\avr-\avl|},
	\end{align*}
	and
	\begin{align*}
		|\mathcal{A}_2'(\Y_0)-1 |\leq \frac{\norm{\psi_{0l},\psi_{0r}}_{L^{\infty}(\R)}}{|\aur-\aul|},
	\end{align*}
	we can apply the implicit function theorem to uniquely determine $(\X_0,\Y_0)$ satisfying \eqref{XY1} provided $\varepsilon_0$ small enough. 
	
	Then, thanks to Lemma \ref{decay} and $\varepsilon_0$ small enough, the existence and uniqueness of $(\X,\Y)$ can follow from the Cauchy-Lipschitz theorem. Moreover, the
	exponential decay of $\X',\Y'$ holds.
	
	Finally, $\X_\infty,\Y_\infty$ can be obtained by the integration by parts over the special domain depending $\X$ or $\Y$ respectively, for example, considering, for any fixed $x\in[0,1],t>0$ and integer $N>0$, 
	\begin{align*}
		\Omega ^N_{x,t}:=\{(y,\tau):0<\tau<t,\quad \Gamma ^N_l(\tau)<y<\Gamma ^N_r(\tau)\}
	\end{align*}
	with $\Gamma ^N_l(\tau):=s\tau+\Y(\tau)+(-N+x)\pi_l,\quad
		\Gamma ^N_r(\tau):=s\tau+\Y(\tau)+(N+x)\pi_r.$ 
\end{pf}

Note that as $t \to +\infty$, the ansatz $(\vt,\ut)$ in \eqref{ansatz} tends to $(v^S(x-st-\X_{\infty}),u^S(x-st-\Y_{\infty}))$. This pair is a
traveling wave solution to \eqref{VNC-Ion} if and only if 
\begin{equation}\label{x=y}
	\X_{\infty}=\Y_{\infty},
\end{equation}
which unfortunately does not hold for generic initial perturbations. In order to ensure
that \eqref{x=y} holds, we assume additionally that the initial data $(v_0,u_0)$ satisfies \eqref{x=y}. In fact,
by plugging \eqref{XY1} into \eqref{C1} and \eqref{C2}, respectively,  \eqref{x=y} is equivalent to the following zero-mass type condition
\begin{equation}\label{new0mass}
	\begin{aligned}
		&s\left\{\int_{-\infty}^{0}(v_0-v^S-\phi_{0l})(x)\,dx+\int_{0}^{+\infty} (v_0-v^S-\phi_{0r})(x)\,dx\right.\\
		&\left.-\frac{1}{\pi_l}\int_{0}^{\pi_l}\int_{0}^{x} \phi_{0l}(y) \,dy \,dx+\frac{1}{\pi_r}\int_{0}^{\pi_r}\int_{0}^{x} \phi_{0r}(y) \,dy \,dx\right\}\\
		&=-\int_{-\infty}^{0}(u_0-u^S-\psi_{0l})(x)\,dx-\int_{0}^{+\infty}(u_0-u^S-\psi_{0r})(x)\,dx\\
		&+\frac{1}{\pi_l}\int_{0}^{\pi_l}\int_{0}^{x} \psi_{0l}(y) \,dy \,dx+\int_{0}^{+\infty}\frac{1}{\pi_l}\int_{0}^{\pi_l}[\sigma(v_l(x,t))-\sigma(\avl)] \,dx \,dt\\
		&-\frac{1}{\pi_r}\int_{0}^{\pi_r}\int_{0}^{x} \psi_{0r}(y) \,dy \,dx-\int_{0}^{+\infty}\frac{1}{\pi_r}\int_{0}^{\pi_r}[\sigma(v_r(x,t))-\sigma(\avr)] \,dx \,dt.
	\end{aligned}
\end{equation}
It should be remarked that when all the periodic perturbations $\phi_{0l},\phi_{0r},\psi_{0l}$ and $\psi_{0r}$ vanish, \eqref{new0mass} reduces to
	\begin{align*}
		s\int_{\R}(v_0-v^S)\,dx+\int_{\R}(u_0-u^S)\,dx=0.
	\end{align*}
    This condition is equivalent to that there exists a (unique) constant $x_0 \in \R$ such that 
    \begin{align*}
		\int_{\R} (v_0(x)-v^S(x-x_0)) \,dx=0, \quad  \int_{\R} (u_0(x)-u^S(x-x_0)) \,dx=0,
    \end{align*}
    which is indeed the zero-mass condition given in \cite{Goodman.1986,zero-mass}. 

%i.e.
%\begin{align*}
%	&\Phi_0(x)=\int_{-\infty}^{x}[(v_0-\avl-\phi_{0l})+(\avl-\avr+\phi_{0l}-\phi_{0r})g_{\X_0}](y)\,dy,\\
%	&\Psi_0(x)=\int_{-\infty}^{x}[(u_0-\aul-\psi_{0l})+(\aul-\aur+\psi_{0l}-\psi_{0r})g_{\Y_0}](y)\,dy.
%\end{align*}

\subsection{Main Result}
With the ansatz constructed above, we define the anti-derivative of the difference
between the initial data \eqref{initial} and ansatz,
\begin{equation}\label{initialansatz}
	\Phi_0(x):=\int_{-\infty}^{x}(v_0(y)-\vt(y,0))  \,dy,\quad \Psi_0(x):=\int_{-\infty}^{x}(u_0(y)-\ut(y,0))  \,dy.
\end{equation}
Set
\begin{equation}\label{E0}
	E_0:=\norm{\phi_{0l},\psi_{0l}}_{H^3(0,\pi_l)}+\norm{\phi_{0r},\psi_{0r}}_{H^3(0,\pi_r)}+\norm{\Phi_0}_{H^2(\R)}+\norm{\Psi_0}_{H^2(\R)}.
\end{equation}
Then, the main result of this paper can be stated as follows.
\begin{theorem}\label{mainresult}
	Assume that the periodic perturbations in \eqref{cond-farfield} satisfy \eqref{0-A} and \eqref{new0mass}. Then there exists $\varepsilon_0>0$ such that
	, if $E_0\leq \varepsilon_0$, there admits a unique global solution of \eqref{VNC-Ion} and \eqref{initial}, satisfying
\begin{equation}\label{regularity}
		\begin{aligned}
		v-\vt \in C\Big([0,+\infty);H^1(\R)\Big)\cap L^2\Big((0,+\infty);H^1(\R)\Big),\\
		u-\ut \in C\Big([0,+\infty);H^1(\R)\Big)\cap L^2\Big((0,+\infty);H^2(\R)\Big),
	\end{aligned}
\end{equation}
	and 
	\begin{equation}\label{asymptotic}
		\norm{(v,u)(\cdot,t)-(v^S,u^S)(\cdot-st-\X_{\infty})}_{L^{\infty}(\R)}\to 0 \quad \text{as}~t \to +\infty,
	\end{equation}
	with the constant shift $\X_{\infty}(=\Y_{\infty})$ given by
	\begin{align*}
		\X_{\infty}=-\frac{1}{\avr-\avl}\left\{\int_{-\infty}^{0}(v_0-v^S-\phi_{0l})(x)  \,dx+\int_{0}^{+\infty}(v_0-v^S-\phi_{0r})(x)  \,dx  \right\}+\mathcal{C}_1,
	\end{align*}
	where $\mathcal{C} _1$ is given in \eqref{C1}.
\end{theorem}
\begin{remark}
	In Theorem \ref{mainresult}, the amplitude of the shock wave can be arbitrarily large.
\end{remark}
\begin{remark}
	When the periodic perturbations vanish, the result of Theorem \ref{mainresult} implies  the stability of large-amplitude shock to the nonconvex system of viscoelasticity \eqref{VNC-Ion}, which remove the small amplitude assumption of shock in \cite{InitialModel} and the stronger than Lax's shock condition in \cite{Gaijin}.
\end{remark}
\begin{remark}
	The corresponding result as Theorem \ref{mainresult} for the degenerate shock case in \eqref{degenerate} remains open and be left for the further study.
\end{remark}
\begin{remark}
	By the similar proof of the a priori estimates below, we can show that Theorem \ref{mainresult}  also holds for the following non-degenerate case
	\begin{equation*}
		\begin{aligned}
			\sigma''(v)\gtrless 0,\quad \text{for} ~~ v\lessgtr 0 ~~\text{under consideration,}
		\end{aligned}
	\end{equation*}
	\begin{equation*}
		\begin{aligned}
		    \sigma'''(v)<0, \quad \text{for}~ v\neq0~\text{under consideration,}
	    \end{aligned}
	\end{equation*}
	which is stated as an open problem in \cite{InitialModel} for the stability of shock without periodic perturbation.
\end{remark}
\section{The a prior estimates}\label{section3}
%In this section, we first prove the decay of the error term of our new system, then we consider the effective velocity and introduce the construction of transform function and 
%weight funtion, by using suitable weight function and transform function for new system, we prove the basic estimate. Then we can finish the energy estimate 
%of high-order terms.

The key ingrdient to prove Theorem \ref{mainresult} is to derive the a prior estimates for solutions to the perturbation equations in terms of the anti-derivative variables. For convenience, we change the coordinates from $(x,t)$ to $(\xi,t)$ with $\xi=x-st$ and denote 
\begin{align*}
	(\mathbf{v},\mathbf{u})(\xi,t):=(v,u)(\xi+st,t), \quad (\tilde{\mathbf{v}},\tilde{\mathbf{u}})(\xi,t)=(\vt,\tilde{u})(\xi+st,t).
\end{align*}
Define the anti-derivative variables of  perturbation $(\mathbf{v}-\mathbf{\tilde{v}},\mathbf{u}-\mathbf{\tilde{u}})$ by
\begin{equation}\label{adv}
	\Phi(\xi,t):=\int_{-\infty}^{\xi}(\mathbf{v}-\mathbf{\tilde{v}})(y,t) \,dy ,\quad \Psi(\xi,t):=\int_{-\infty}^{\xi}(\mathbf{u}-\mathbf{\tilde{u}})(y,t) \,dy
\end{equation}
Taking the difference of  \eqref{VNC-Ion} and \eqref{VNC-T} in $(\xi, t)$-coordinate and integrating the resulting equation on $(-\infty,\xi)$ yield that 
\begin{equation}\label{system2}
	\begin{cases}
		\p_t\Phi-s\p_{\xi}\Phi-\p_{\xi}\Psi=\mathbf{H}_1,\\
		\p_t\Psi-s\p_{\xi}\Psi-\sigma'(\mathbf{\tilde{v}})\p_{\xi}\Phi=\sigma(\mathbf{v}|\mathbf{\tilde{v}})+\mu \p^2_{\xi}\Psi+\mathbf{H}_2,
	\end{cases}
\end{equation}
where
\begin{equation}\label{hterm}
	\sigma(\mathbf{v}|\mathbf{\tilde{v}}):=\sigma(\mathbf{v})-\sigma(\mathbf{\tilde{v}})-\sigma'(\mathbf{\tilde{v}})\p_{\xi}\Phi\geqq 0
\end{equation}
and 
\begin{equation}\label{h1h2}
	\mathbf{H}_1:=\mathbf{F}_1+\mathbf{F}_2, \quad \mathbf{H}_2:=\mathbf{F}_3+\mathbf{F}_4
\end{equation}
with $\mathbf{F}_1(\xi,t)=F_1(\xi+st,t)$ and $\mathbf{F}_3(\xi,t)=F_3(\xi+st,t)$ given by \eqref{VNC-T}, and 
\begin{align*}
	\mathbf{F}_2(\xi,t):=\int_{-\infty}^{\xi}\mathbf{f}_2(y,t)  \,dy=\int_{-\infty}^{\xi}f_2(y+st,t)  \,dy=-\int_{\xi}^{+\infty}f_2(y+st,t)  \,dy,\\
	\mathbf{F}_4(\xi,t):=\int_{-\infty}^{\xi}\mathbf{f}_4(y,t)  \,dy=\int_{-\infty}^{\xi}f_4(y+st,t)  \,dy=-\int_{\xi}^{+\infty}f_4(y+st,t)  \,dy.
\end{align*}
The initial data for $(\Phi,\Psi)$ are given by
\begin{equation}\label{newinitial}
	(\Phi,\Psi)(\xi,0)=(\Phi_0,\Psi_0)(\xi) \in H^2(\R) \times H^2(\R),
\end{equation}

We will derive the following key proposition of a prior estimate for the solution $(\Phi,\Psi)$ to \eqref{system2} and \eqref{newinitial}. 
\begin{proposition}\label{mainpro}
	(A priori estimates). For any $\T>0$, assume that $(\Phi,\Psi)$ is a smooth solution to \eqref{system2} and \eqref{newinitial}. Then there
	exist $\varepsilon_0>0$ and $\delta_0>0$, independent of ~$\T$, such that if 
	\begin{equation}\label{initialsmall}
		\varepsilon=E_0<\varepsilon_0~~\text{and}~~\delta=\mathop{sup}\limits_{t \in (0,\T)}\norm{(\Phi,\Psi)}_2(t)<\delta_0,
	\end{equation}
	then
	\begin{equation}\label{proesti}
		\mathop{sup}\limits_{t \in [0,\T]}\norm{(\Phi,\Psi)}^2_2(t)+\int_{0}^{\T}(\norm{\p_{\xi}\Phi}^2_1+\norm{\p_{\xi}\Psi}^2_2)(t)  \,dt\leq C_0
		(\norm{(\Phi_0,\Psi_0)}^2_2+\varepsilon),
	\end{equation}
	where $C_0>0$ is independent of $\varepsilon$,$\delta$ and $\T$.
\end{proposition}
In order to prove Proposition \ref{mainpro}, on the one hand, we
follow the idea of \cite{Yao.2016, HehuangJDE, HuangYuanCMP}to consider the effective velocity, $\mathbf{m}:=\mathbf{u}-\mu \p_{\xi}\mathbf{v}$, instead of 
the velocity $\mathbf{u}$, to obtain the energy estimate of lower-order terms even for the large-amplitude shock. On the other hand, to overcome the difficulty from non-convexity, we construct
a weight function and a transform function to obtain the energy estimate of lower-order terms for the equivalent new system \eqref{transystem}, see Lemma \ref{lowterm1} and \ref{lowterm2}. Finally, we turn back to \eqref{system2} 
to obtain the higher-order estimate, see Lemma \ref{lowterm3}.

Define $\mathbf{\tilde{m}}:=\mathbf{\tilde{u}}-\mu \p_{\xi}\mathbf{\tilde{v}}$, and denote the anti-derivative variable of the perturbed effective velocity as
\begin{equation}\label{w}
	W(\xi,t):=\int_{-\infty}^{\xi}(\mathbf{m}-\mathbf{\tilde{m}})(y,t)  \,dy =\Psi(\xi,t)-\mu \p_{\xi}\Phi.
\end{equation}
Then, it follows from \eqref{system2} that
\begin{equation}\label{mainsystem}
	\begin{cases}
		\p_t\Phi-s\p_{\xi}\Phi-\p_{\xi}W=\mu \p^2_{\xi}\Phi + \mathbf{H}_1,\\
		\p_tW-s\p_{\xi}W-\sigma'(\mathbf{\tilde{v}})\p_{\xi}\Phi=\sigma(\mathbf{v}|\mathbf{\tilde{v}})+\mathbf{H}_2-\mu \p_{\xi}\mathbf{H}_1,
	\end{cases}
\end{equation}
The initial datum of $W$ is 
\begin{align*}
	W_0(\xi):=W(\xi,0)=\Psi_0(\xi)-\mu\p_\xi \Phi_0 \in H^1(\R),
\end{align*}
which satisfies that $\norm{W_0}_{H^1(\R)}\leq \norm{\Psi_0}_{H^1(\R)}+C\norm{\Phi_0}_{H^2(\R)} \leq C\varepsilon$.

Under the assumption that $\varepsilon_0$ and $\delta_0$ are both small enough, it follows from \eqref{initialsmall} and Sobolev inequality that 
\begin{equation}\label{small1}
	\begin{aligned}
		\mathop{sup}\limits_{t \in [0,\T]}\norm{(\Phi,\Psi)}_{W^{1,\infty}(\R)}\leq C\mathop{sup}\limits_{t \in [0,\T]}\norm{(\Phi,\Psi)}_2\leq C\delta_0,\\
		\mathop{sup}\limits_{t \in [0,\T]}\norm{W}_{L^{\infty}(\R)}\leq C\mathop{sup}\limits_{t \in [0,\T]}\norm{W}_1\leq C\delta_0,
	\end{aligned}
\end{equation}
and 
\begin{equation}\label{v-bdd}
	\frac{\avl}{2}\leq \mathbf{\tilde{v}}(\xi,t)\leq 2\avr, \quad \frac{\avl}{4}\leq \mathbf{v}(\xi,t)\leq 4\avr,\text{ for }\xi \in \R, t\in[0,T].
\end{equation}

The following decay properties of the
error terms in \eqref{system2} or \eqref{mainsystem} are useful in the a prior estimates, whose proof is the same, so that being omitted here, as the one in \cite{HuangYuanCMP}. The proof can also follow from an abstract lemma developed in \cite{MeiYuan2025} for the decay estimates of error terms in non-localized perturbation problems.
\begin{lemma}\label{decay-error}
	Under the assumptions of Theorem \ref{mainresult}, it holds that
	\begin{align*}
		\norm{\mathbf{H}_1}_2(t),\norm{\mathbf{H}_2}_1(t)\leq C\varepsilon e^{-\alpha t},~t\geqq 0,
	\end{align*}
	where $C>0$ is independent of $\varepsilon$ and $t$, and $\alpha>0$ is the constant given in Lemma \ref{XYdecay}.
\end{lemma}
\subsection{Lower-order Energy Estimates}\label{3.1sec}
The following lemma gives the lower order estimates of solution $(\Phi, W)$ to the system \eqref{mainsystem}. 
\begin{lemma}\label{lowterm1}
	Under the assumptions of Proposition \ref{mainpro}, there exist $\varepsilon_0>0$ and $\delta_0>0$ such that if $\varepsilon<\varepsilon_0$ and $\delta<\delta_0$, then
	\begin{equation}\label{estimate1}
		\begin{aligned}
			\mathop{sup}\limits_{t \in [0,\T]}\norm{(\Phi,W)}^2(t)+\int_{0}^{\T} \norm{\p_{\xi}\Phi}^2(t) \,dt \leq C(\norm{(\Phi_0,W_0)}^2+\varepsilon),
		\end{aligned}
	\end{equation}
	where $C>0$ is independent of $\varepsilon,\delta$ and $\T$.
\end{lemma}
The proof of lemma \ref{lowterm1} should not follow the classical energy method as in \cite{HuangYuanCMP} did for the compressible Navier-Stokes equations
with convex pressure. This is because that the bad term $\frac{s\sigma''(\vsx)(\vsx)'}{\sigma'(\vsx)}W^2$ would appear in the energy estimates and change the sign, when $\sigma(v)$ is non-convex. This difficulty of non-convexity was overcame in Kawashima and Matsumura \cite{InitialModel} by weighted energy estimate method for the stability of the viscous shock with a small amplitude to the system \eqref{VNC-Ion}-\eqref{SC2}. However, this weighted energy estimate method can not be directly applied to the large amplitude shock studied here. 
When $s$ is assumed to satisfy the condition (condition (2.18) in \cite{Gaijin}):
\begin{equation}\label{mei's condition}
	s^2<\sigma'(\avl)+\frac{1}{2}\sigma''(\avl)[2(v^*-\avr)+\avl-\avr],
\end{equation}
Matsumura and Mei \cite{Gaijin} proved the stability of the viscous shock profile with non-convex relation of the opposite sign to \eqref{SC2} and \eqref{SC3}by introducing a transform function $T(v)$ and using weighted energy estimates for the transformed system. However, a transform function $T(v)$ introduced in  \cite{Gaijin} is not adaptable for the Lax's shock condition \eqref{non-degenerate} here even for the non-convex $\sigma(v)$ with the sign \eqref{SC2}-\eqref{SC3}. To remove the condition \eqref{mei's condition} and get the stability of shock with large amplitude, we have to introduce the system \eqref{mainsystem} of $(\Phi, W)$, where $W$ is the anti-derivative of the perturbed effective velocity in \eqref{w}, and modify the construction of transform and weighted functions in \cite{Gaijin} based on the structure of transformed system of \eqref{mainsystem}.

To show our idea more precisely, let us briefly introduce the idea of  Matsumura and Mei in \cite{Gaijin}. Under the shock speed condition \eqref{mei's condition}, they constructed the transform function $T(v^S)$ and weighted one $w(v^S)$, which are both $C^1([\bar{v}_l,\bar{v}_r])$, and performed the weighted energy estimates for the transformed system of  $(\ph,\mathbf{\Psi})$, which are defined by $$(\Phi,\Psi)(t,\xi)=T(v^S(\xi))(\ph(t,\xi),\mathbf{\Psi}(t,\xi)),$$ 
to obtain the following weighted energy identity
\begin{equation}\label{mei1}
	\begin{aligned}
		&\frac{1}{2}\p_t\left\{(w\sigma')(v^S)\ph^2+w(v^S)\mathbf{\Psi}^2\right\}-\p_{\xi}\left\{\dots\right\}+\mu w(v^S)\p_{\xi}\mathbf{\Psi}^2\\
		&+\frac{|\p_{\xi}v^S|}{2s}w(v^S)Y(v^S)(s\ph+\mathbf{\Psi})^2+\frac{|\p_{\xi}v^S|}{s}w(v^S)Z(v^S)\mathbf{\Psi}^2=\cdots,
	\end{aligned}
\end{equation}
for some $Y(v^S)$ and $Z(v^S)$ (c.f. $(4.8)-(4.9)$ in \cite{Gaijin}) satisfying
 \begin{equation}\label{YZ-prope}
 	Y(v^S)\geqq0,~~Z(v^S)\geqq C|\sigma''(v^S)|\text{ in }[\avl,\avr],
 \end{equation}
  whenever \eqref{mei's condition} is assumed. However, this property of $Y(v^S)$ and $Z(v^S)$ in \eqref{YZ-prope} can not be guaranteed by the $T$ and $w$ constructed in \cite{Gaijin} when the amplitude of shock is large. 
  
  Inspired by \cite{Yao.2016, HehuangJDE, HuangYuanCMP} and \cite{Gaijin}, we will apply the method of transformed and  weighted energy estimates to the system \eqref{mainsystem} so that we can obtain the weighted energy estimate with modified $Y(v^S)$ and $Z(v^S)$ and an extra term $X(v^S)$ in \eqref{X-form}-\eqref{Z-form} satisfying 
  \begin{equation}\label{XYZ-cond}
  	X(v^S)\geqq 0, ~~Y(v^S)\geqq0,~~Z(v^S)\geqq C|\sigma''(v^S)|\text{ in }[\avl,\avr].
  \end{equation}
     To achieve this, we have to construct new transform and weighted functions. However, it seems difficult to construct $C^1$ transform and weighted functions in the whole interval $[\bar{v}_l,\bar{v}_r]$ for the system \eqref{mainsystem} with non-degenerate shock in \eqref{non-degenerate}. In particular, when $\vsx \in [0,\avr]$ and no restriction of the shock speed $s$ like \eqref{mei's condition},  $Z(\vsx)$ defined in \eqref{Z-form} may not satisfy \eqref{YZ-prope} because $\sigma''(\vsx)$ could be very large and out of control. Our strategy is to divide the interval $[0,\avr]$ into finite sub-intervals  and construct transform and weighted functions in each intervals so that  $Z(\vsx)\geqq C|\sigma''(v^S)|$ in each sub-intervals. But this method only works for $\vsx$ away from $\avr$, because near the point $\avr$, $X(\vsx)$ may be negative. Therefore, we firstly take advantage of the property that $((\vsx)')^2=(\frac{h(\vsx)}{s\mu} )^2\to 0$ is faster than $\frac{h(\vsx)(s^2-\sigma'(\vsx))}{(s\mu)^2}$ as $\vsx \to \avr$ to construct weight and 
transform functions in the interval $[v_b,v_r] $ with $v_b$ close enough to $\avr$, so that both $Z(\vsx)\geqq C|\sigma''(v^S)|$ and $X(\vsx)\geqq0$ holds true on $[v_b,v_r] $. As a cost of our strategy, our weighted and transform functions are no longer $C^1$, but only piecewise smooth functions. Therefore, 
our construction of weight and transform functions would ensure that the jump of flux $\zeta$, defined by \eqref{zeta} below, on each non-smooth point is non-positive in the process of 
energy estimate. In the following, we will construct the weighted and transform functions separately in  $[\avl,0]$ and $[0,\avr]$ and extend them a little bit to $[\avl,v_d]$ and $[v_{-1},\avr]$, and divide the energy estimates on $\mathbb{R}$ into two parts with
$\vsx$ in $[\avl,0]$ and $[0,\avr]$.
It should be remarked that the extension of our weighted and transform functions is used to get the uniform lower bounds of $(\ph,\w)$ and make the jump of flux $\zeta$ non-positive.

\subsubsection{A new system after transformation}

\

Let $T(\vsx(\xi))$ be a transform function  to be constructed below, where $\vsx(\xi)$ is the viscous shock profile with the shift $\mathcal{X}(t)$. Define 
\begin{equation}\label{tran}
	(\Phi,W)(\xi,t)=T(\vsx(\xi))(\ph(\xi,t),\w(\xi,t)).
\end{equation}
Then the system \eqref{mainsystem} can be transformed into a new one in terms of $(\ph,\w)$ as
\begin{equation}\label{transystem}
	\begin{cases}
		\p_t\ph+\frac{\p_tT}{T}\ph-s\p_{\xi}\ph-\p_{\xi}\w-s\frac{\p_{\xi}T}{T}\ph-\frac{\p_{\xi}T}{T}\w-2\mu\frac{\p_{\xi}T}{T}\p_{\xi}\ph-\mu\p^2_{\xi}\ph-\mu\frac{\p^2_{\xi}T}{T}\ph=\frac{\mathbf{H}_1}{T},\\
		\p_t\w+\frac{\p_tT}{T}\w-s\p_{\xi}\w-\sigma'(\vsx)\p_{\xi}\ph-s\frac{\p_{\xi}T}{T}\w-\sigma'(\vsx)\frac{\p_{\xi}T}{T}\ph\\
		\qquad\qquad\qquad\qquad\qquad=\frac{\sigma(\mathbf{v}|\mathbf{\tilde{v}})+\mathbf{H}_2-\mu\p_{\xi}\mathbf{H}_1}{T}+(\sigma'(\mathbf{\tilde{v}})-\sigma'(\vsx))(\p_{\xi}\ph+\frac{\p_{\xi}T}{T}\ph),
	\end{cases}
\end{equation}
where $\p_tT=\frac{\partial(T(\vsx(\xi)))}{\partial t}$, $\p_\xi T=\frac{\partial(T(\vsx(\xi)))}{\partial \xi}$ and 
$\p_\xi^2T=\frac{\partial^2(T(\vsx(\xi)))}{\partial\xi^2}$. 
We will derive the weighted energy estimates for the transformed system \eqref{transystem}. To do so, let $w(\vsx)$ be a weighted function to be chosen below, and let $a(\vsx):=\sigma'(\vsx)w(\vsx)$. Multiplying the first and second equation in \eqref{transystem} by $a(\vsx)\ph$ and $w(\vsx)\w$ respectively and summing the resulting equations, we can get
\begin{equation}\label{transe}
	\begin{aligned}
		&\p_t\left(\frac{a(\vsx)}{2}\ph^2+\frac{w(\vsx)}{2}\w^2\right)+\mu a(\vsx)(\p_{\xi}\ph)^2\\
		&+\mu X(\vsx)\ph^2+\frac{a(\vsx)}{s}(\vsx)'Y(\vsx)(s\ph+\w)^2+(\vsx)'Z(\vsx)\w^2\\
		&=\underbrace{\frac{\X'(\vsx)'(\ph^2(2a(\vsx)T'(\vsx)-a'(\vsx)T(\vsx))+\w^2(2w(\vsx)T'(\vsx)-w'(\vsx)T(\vsx)))}{2T(\vsx)}}_{\mathbf{I}_{1,1}}\\
		&+\underbrace{\frac{\sigma(\mathbf{v}|\mathbf{\tilde{v}})w(\vsx)\w+(\mathbf{H}_2-\mu\p_{\xi}\mathbf{H}_1)w(\vsx)\w}{T(\vsx)}}_{\mathbf{I}_{1,2}}+\underbrace{\frac{\mathbf{H}_1\ph a(\vsx)}{T(\vsx)}}_{\mathbf{I}_{1,3}}+\underbrace{(\sigma'(\mathbf{\tilde{v}})-\sigma'(\vsx))w(\vsx)\p_{\xi}\ph\w}_{\mathbf{I}_{1,4}}\\
		&+\underbrace{\frac{T'(\vsx)}{T(\vsx)}(\sigma'(\mathbf{\tilde{v}})-\sigma'(\vsx))w(\vsx)(\vsx)'\ph\w}_{\mathbf{I}_{1,5}}+\underbrace{\p_{\xi}\zeta(\xi)}_{\mathbf{I}_{1,6}},\\
	\end{aligned}
\end{equation}
where
\begin{align}\label{X-form}
	&X(\vsx):=\p_{\xi}T\p_{\xi}\left(\frac{a}{T}\right)-\frac{1}{2}\p_{\xi}^2a\nonumber\\
	=&\left(\frac{T'(\vsx)}{T(\vsx)}a'(\vsx)-\left(\frac{T'(\vsx)}{T(\vsx)}\right)^2a(\vsx)-\frac{a''(\vsx)}{2}\right)\frac{h^2(\vsx)}{s^2\mu^2}-\frac{a'(\vsx)h(\vsx)h'(\vsx)}{2s^2\mu^2}
\end{align}
with $h(\vsx)$ being defined in Lemma \ref{unique}.
\begin{equation}\label{Y-form}
	Y(\vsx):=\frac{a'(\vsx)}{2a(\vsx)}-\frac{T'(\vsx)}{T(\vsx)}.
\end{equation}
and 
\begin{align}\label{Z-form}
	Z(\vsx):=&\frac{sw'(\vsx)}{2}-s\frac{T'(\vsx)}{T(\vsx)}w(\vsx)+\frac{T'(\vsx)}{sT(\vsx)}a(\vsx)-\frac{a'(\vsx)}{2s}\nonumber\\
	=&sw(\vsx)\left(\frac{s^2-\sigma'(\vsx)}{s^2}Y(\vsx)-\frac{\sigma''(\vsx)}{2\sigma'(\vsx)}\right),
\end{align}
Moreover, the flux $\zeta(\xi)$ is given by 
\begin{align}\label{zeta}
	\zeta(\xi):=&\frac{s a(\vsx)}{2}\ph^2+\frac{sw(\vsx)}{2}\w^2+\w\ph a(\vsx)+\mu a(\vsx)\p_{\xi}\ph\ph-\mu a(\vsx)Y(\vsx)(\vsx)'\ph^2.
\end{align}
%\begin{align}
%	\zeta(\xi):=&\frac{s a(\vsx)}{2}\ph^2+\frac{sw(\vsx)}{2}\w^2+\w\ph a(\vsx)+\mu\frac{T'(\vsx)}{T(\vsx)}a(\vsx)(\vsx)'\ph^2\nonumber\\
%    &-\frac{\mu}{2}a'(\vsx)(\vsx)'\ph^2.
%\end{align}
\subsubsection{Construction of Weight Function and Transform Function}\label{construction of function}

\

In order to get weighted energy estimates, we have to construct suitable $T(\vsx)$ and $w(\vsx)$ such that $X(\vsx)$, $Y(\vsx)$ and $Z(\vsx)$ in \eqref{transe} satisfying \eqref{XYZ-cond}. As explained in the paragraph below Lemma \ref{lowterm1}, we will construct separately the weighted and transform functions in $[0,\avr]$ and $[\avl,0]$. We firstly construct the weight and transform functions of $\vsx$ in $[v_b,\avr]$ with $v_b$ sufficiently close to $\avr$, so that \eqref{XYZ-cond} holds for $\vsx\in [v_b,\avr]$. Secondly, we divide $[0,v_b]$ into finite sub-intervals and construct the weighted and transform functions in each sub-interval so that \eqref{XYZ-cond} are fulfilled. Furthermore, we also need to slightly extend the weighted and transform functions to $[v_{-1},\avr]$ so that the jump of flux $\zeta$ on each non-smooth points vanishes in the process of energy estimates. Finally, we construct the other transform and weight functions in $[\avl,0]$ and also slightly extend them to $[\avl,v_d]$ with $v_d>0$ so that the jump of flux $\zeta$ on each non-smooth points vanishes in the process of energy estimates.

\begin{part}{Construction of $T$ and $w$ on $[0,\avr]$}

First we introduce the construction of $T(\vsx)$ and $w(\vsx)$ on the interval $[0,\avr]$.

 For the non-degenerate shock \eqref{non-degenerate}, when  $\vsx\in[v_b,\avr]$ for some $v_b$, which will be determined in \eqref{XVb}, close to $\avr$, we define the weighted function 
 \begin{equation}
 	w(\vsx)=\frac{C_0}{(s^2-\sigma'(\vsx))^2},\quad \vsx\in[v_b,\avr]
 \end{equation}
and the transform function
\begin{equation}
	T(\vsx)=C_b,\quad \vsx\in[v_b,\avr]
\end{equation}
where $C_0,C_b$ are constants to be given exactly later. Then, it is obvious that
\begin{equation}
	Y(\vsx)=\frac{\sigma''(\vsx)(s^2+\sigma'(\vsx))}{2\sigma'(s^2-\sigma'(\vsx))}\geq 0
\end{equation}
and 
\begin{equation}
	Z(\vsx)=\frac{w\sigma''(\vsx)}{2s}\geq \frac{C_0}{2s(s^2-\sigma'(\avl))^2}\sigma''(\vsx).
\end{equation}
For $X(\vsx)$ on the interval $[v_b,\avr]$, we have 
\begin{equation}\label{part11}
	X(\vsx)=\frac{C_0}{s^2\mu^2(s^2-\sigma'(\vsx))^2}\left[(s^2+\sigma'(\vsx))\sigma''(\vsx)h(\vsx)- \frac{\vartheta(\vsx)}{(s^2-\sigma'(\vsx))^2}h^2(\vsx)\right],
\end{equation}
where
\begin{align*}
	\vartheta(\vsx):=[\sigma'''(\vsx)(s^2+\sigma'(\vsx))+(\sigma''(\vsx))^2](s^2-\sigma'(\vsx))+3(\sigma''(\vsx))^2(s^2+\sigma'(\vsx)).
\end{align*}
$\vartheta(\vsx)$ is bounded for $\vsx \in [\avl,\avr]$. Since $h(\vsx)\to 0$ as $\vsx \to \avr$, there exist a $v_b\in(0,\avr)$  such that  
\begin{equation}\label{XVb}
	X(\vsx)\geq 0, \quad \vsx\in[v_b,\avr].
\end{equation}
for the non-degenerate shock \eqref{non-degenerate}.

Next, we construct  $T(\vsx)$ and $w(\vsx)$ on $[0,v_b]$. 

Let $H>1$ be a constant to be chosen, and assume $N:=[v_b\cdot 2H]+1$ is even with loss of generality. Set  $v_i=\frac{i}{2H}(i=0,1,2\dots,N-1), v_N=v_b$. Then, $[0,v_b]$ can be divided into $N$ intervals $[v_i,v_{i+1}], i=0,\cdots, N-1$.  Define
	\begin{align}
	\bar{a}(\vsx):=\bar{w}(\vsx)\sigma'(\vsx)=
	\begin{cases}
		[H(\frac{1}{2H}+\vsx-v_{i})]^{\frac{5}{4}},&\vsx \in [v_{i},v_{i+1}], ~i=2k-2,\\
		[H(\frac{1}{H}+v_{i}-\vsx)]^{\frac{5}{4}},&\vsx \in [v_{i},v_{i}],~ i=2k-1,
	\end{cases}
\end{align}
for $k=1,2,\dots,\frac{N}{2}$. We further extend the definition of $a(\vsx)$ to the interval $[v_{-1}, v_0]$ with $v_{-1}=-\frac{1}{2H}$ as
\begin{equation}
	\bar{a}(\vsx):=\bar{w}(\vsx)\sigma'(\vsx)=[H(\frac{1}{2H}+\vsx)]^{\frac{5}{4}},~~\vsx \in[v_{-1},v_0].
\end{equation}
To ensure $s^2-\sigma'(\vsx)>0$ on $[v_{-1},v_0]$, we choose $H>-\frac{1}{2v^\ast}$, where $v^\ast<0$ is the constant in Lemma \ref{4cond}. 
The graph of $\bar{a}(\vsx)$ is shown as follows
\begin{figure}[htb]\label{abar}
	\begin{center}
		\begin{tikzpicture}
			\draw[->] (0,0)--(8.5,0);
			\draw[->] (0,0)--(0,4);
			\draw[thick,->] (-1.5,0)--(8.5,0) node[anchor=north west] {$\vsx$};
			\draw[thick,->] (0,0)--(0,4) node[anchor=north east] {$\bar{a}$};
			\node[below] at (0.1,0){0=$v_0$};
			\node[below] at (-1,0){$v_{-1}$};
			\node[below] at (1,0){$v_1$};
			\node[below] at (2,0){$v_2$};
			\node[below] at (3,0){$v_3$};
			\node[below] at (4,0){$v_4$};
			\node[below] at (5,0){$v_5$};
			\node[below] at (6,0){$v_6$};
			\node[below] at (7,0){$\dots$};
			\node[below] at (8,0){$v_N$=$v_b$};
			\node[left] at (0,1){1};
			\foreach \x in {0,1,...,9}
			{
				\draw[xshift=\x cm] (-1,0) -- (-1,0.1);
			};
			\foreach \x in {0}
			{
				\draw[yshift=\x cm] (0,1) -- (0.1,1);
			};
			\draw[domain=-1:1] plot(\x,{(1/2+\x/2)^(5/4)});
			\draw[domain=1:2] plot(\x,{(3/2-\x/2)^(5/4)});
			\draw[domain=2:3] plot(\x,{(\x/2-1/2)^(5/4)});
			\draw[domain=3:4] plot(\x,{(5/2-\x/2)^(5/4)});
			\draw[domain=4:5] plot(\x,{(\x/2-3/2)^(5/4)});
			\draw[domain=5:6] plot(\x,{(7/2-\x/2)^(5/4)});
			\draw[domain=6:7] plot(\x,{(\x/2-5/2)^(5/4)});
			\draw[domain=7:8] plot(\x,{(9/2-\x/2)^(5/4)});
		\end{tikzpicture}
	\end{center}
	\caption{Graph of $\bar{a}(\vsx)$}
\end{figure}

Now we choose the parameter $H>\max\{1,-\frac{1}{2v^\ast}\}$ properly to ensure $X(\vsx)$, $Y(\vsx)$ and $Z(\vsx)$ satisfy \eqref{XYZ-cond} on each interval $[v_i,v_{i+1}], i=0,\cdots, N-1$. 

Note that, for $k=1,2,3,\dots,\frac{N}{2},$
	\begin{align*}
	\bar{a}'(\vsx)=
	\begin{cases}
		\frac{5H}{4}[H(\frac{1}{2H}+\vsx)]^{\frac{1}{4}}\geq 0,&\vsx \in [v_{-1},v_{0}],\\
		\frac{5H}{4}[H(\frac{1}{2H}+\vsx-v_{i})]^{\frac{1}{4}}\geq5\cdot 2^{-\frac{9}{4}}H,&\vsx \in [v_{i},v_{i+1}], ~i=2k-2,\\
		-\frac{5H}{4}[H(\frac{1}{H}+v_{i}-\vsx)]^{\frac{1}{4}}\leq-5\cdot 2^{-\frac{9}{4}}H,&\vsx \in [v_{i},v_{i+1}],~ i=2k-1,
	\end{cases}
\end{align*}
Define $T(\vsx)$ so that
	\begin{align*}
	\frac{T'(\vsx)}{T(\vsx)}=
	\begin{cases}
		\frac{\bar{a}'(\vsx)}{4\bar{a}(\vsx)},&\vsx \in [v_{-1},v_1],\\
		\frac{\bar{a}'(\vsx)}{4\bar{a}(\vsx)},&\vsx \in [v_{i},v_{i+1}], ~i=2k,\\
		\frac{3\bar{a}'(\vsx)}{4\bar{a}(\vsx)},&\vsx \in [v_{i},v_{i+1}], ~i=2k-1,
	\end{cases}
\end{align*}
for $k=1,2,3,\dots,\frac{N}{2},$. Then,  we have
	\begin{equation}\label{tv}
	T(\vsx)=
	\begin{cases}
		\bar{a}(\vsx)^{\frac{1}{4}},&\vsx \in [v_{-1},v_1],\\
		C_{i} \bar{a}(\vsx)^{\frac{3}{4}}, &\vsx \in [v_{i},v_{i+1}], i=2k-1\\
		C_{i} \bar{a}(\vsx)^{\frac{1}{4}}, &\vsx \in [v_{i},v_{i+1}],i=2k,
	\end{cases}
\end{equation}
where the constants $C_i$ are given by
\begin{align*}
C_{i}=\begin{cases}
	\left(\frac{1}{\bar{a}(v_{2k-1})}\prod_{j=1}^{i-2}\frac{\bar{a}(v_{j+1})}{\bar{a}(v_j)}\right)^{\frac{1}{2}},&i=2k-1\\
	\left(\prod_{j=1}^{i-1}\frac{\bar{a}(v_{j+1})}{\bar{a}(v_j)}\right)^{\frac{1}{2}},&i=2k
\end{cases}
\end{align*}
 Therefore, we have
\begin{align*}
	Y(\vsx)=\left(\frac{a'}{2a}-\frac{T'}{T}\right)(\vsx)=\begin{cases}
		\frac{\bar{a}'(\vsx)}{4\bar{a}(\vsx)},&\vsx \in [v_{-1},v_1],\\
		\frac{\bar{a}'(\vsx)}{4\bar{a}(\vsx)},&\vsx \in [v_{i},v_{i+1}], ~i=2k,\\
		-\frac{\bar{a}'(\vsx)}{4\bar{a}(\vsx)},&\vsx \in [v_{i},v_{i+1}], ~i=2k-1,
	\end{cases}
\end{align*}
which yields that 
\begin{equation}
	Y(\vsx)\geq \frac{5}{16}H>0,\text{ for } \vsx\in[v_{-1},v_b].
\end{equation}
To get $Z(\vsx)\geq C|\sigma''(\vsx)|$ on $[v_{-1},v_b]$, it suffice to ensure that
\begin{equation*}
	\frac{s^2-\sigma'(\vsx)}{s^2}Y(\vsx)-\frac{\sigma''(\vsx)}{2\sigma'(\vsx)}\geq \frac{|\sigma''(\vsx)|}{2\sigma'(\vsx)}.
\end{equation*}
When $\vsx\in[0,v_b]$, it follows from \eqref{SC}-\eqref{SC3} that, $s^2-\sigma'(\vsx)\geq s^2-\sigma'(\avr)$, $\sigma''(\vsx)\leq\sigma''(\avr)$, $\sigma'(\vsx)>\sigma'(0)$. Thus, 
\begin{equation*}
	\frac{s^2-\sigma'(\vsx)}{s^2}Y(\vsx)-\frac{\sigma''(\vsx)}{2\sigma'(\vsx)}\geq\frac{s^2-\sigma'(\avr)}{s^2}\frac{5}{16}H\geq \frac{\sigma''(\avr)}{\sigma'(0)}\geq \frac{\sigma''(\vsx)}{\sigma'(\vsx)}
\end{equation*}
when $H$ is chosen large enough so that
\begin{equation}
	H\geq \frac{16s^2\sigma''(\avr)}{5(s^2-\sigma'(\avr))\sigma'(0)}.
\end{equation}
Therefore, we have
\begin{equation}
	Z(\vsx)\geq sw(\vsx)\frac{\sigma''(\vsx)}{\sigma'(\vsx)}\geq C\sigma''(\vsx), \text{ for }\vsx\in[v_0,v_b].
\end{equation} 
When $\vsx\in[v_{-1},0]$, it follows from \eqref{SC2} that $\sigma''(\vsx)<0$ and $s^2-\sigma'(\vsx)\geq 0$. Therefore, 
we have
\begin{equation}
	Z(\vsx)\geq -\frac{\sigma''(\vsx)}{2\sigma'(\vsx)}\geq C|\sigma''(\vsx)|, \text{ for }\vsx\in[v_{-1},v_0].
\end{equation}
It remains to show $X(\vsx)\geq 0$ on $[v_{-1},v_b]$ by choosing sufficiently large $H$. Note that
\begin{align*}
	\bar{a}''(\vsx)=
	\begin{cases}
		\frac{5H^2}{16}[H(\frac{1}{2H}+\vsx-v_{i})]^{-\frac{3}{4}} ,&\vsx \in [v_{i},v_{i+1}], ~i=2k-2,\\
		\frac{5H^2}{16}[H(\frac{1}{H}+v_{i}-\vsx)]^{-\frac{3}{4}},&\vsx \in [v_{i},v_{i+1}],~ i=2k-1,
	\end{cases}
\end{align*}
and 
$$h'(\vsx)=-s^2+\sigma'(\vsx)<0,\text{ for }\vsx\in[v_{-1},v_b]$$
which yields that
\begin{equation*}
	h(\vsx)\geq h(v_b)>h(\avr)=0
\end{equation*}
When $\vsx \in [v_{i},v_{i+1}],~~i=2k-1,k=1,\cdots,\frac{N}{2}$, we have 
\begin{equation}\label{Xodd}
	\begin{aligned}
		X(\vsx)&=\left(\frac{3}{16}\frac{a'^2(\vsx)}{a(\vsx)}-\frac{a''(\vsx)}{2}\right)\frac{h^2(\vsx)}{s^2\mu^2}+\frac{a'(\vsx)h(\vsx)(s^2-\sigma'(\vsx))}{2s^2\mu^2}\\
		&=\frac{35H^2}{256}(H(\frac{1}{H}+v_i-\vsx))^{-\frac{3}{4}}(\frac{h(\vsx)}{s\mu})^2-\frac{5H(H(\frac{1}{H}+v_i-\vsx))^\frac{1}{4}h(\vsx)(s^2-\sigma'(\vsx))}{8(s\mu)^2}\\
		&\geq\frac{h(\vsx)}{s^2\mu^2}\left(\frac{35H^2}{256}h(v_b)-\frac{5H(s^2-\sigma'(\vsx))}{8}\right)\geq0,
	\end{aligned}
\end{equation}
when $H$ is chosen large enough so that
\begin{equation}
	H\geq \frac{8(s^2-\sigma'(0))}{h(v_b)}.
\end{equation}
When $\vsx \in [v_{i},v_{i+1}],~~i=2k,k=1,\cdots,\frac{N}{2}$ or $\vsx\in[v_{-1},v_1]$, we have
\begin{equation}\label{Xeven}
	\begin{aligned}
		X(\vsx)&=\left(\frac{3}{16}\frac{a'^2(\vsx)}{a(\vsx)}-\frac{a''(\vsx)}{2}\right)\frac{h^2(\vsx)}{s^2\mu^2}+\frac{a'(\vsx)h(\vsx)(s^2-\sigma'(\vsx))}{2s^2\mu^2}\\
		&\geq \left(\frac{3}{16}\frac{a'^2(\vsx)}{a(\vsx)}-\frac{a''(\vsx)}{2}\right)\frac{h^2(\vsx)}{s^2\mu^2}\\
		&=\frac{35H^2}{256}(H(\frac{1}{2H}+\vsx-v_i))^{-\frac{3}{4}}(\frac{h(\vsx)}{s\mu})^2\geqq0,
	\end{aligned}
\end{equation}
where we have used $a'(\vsx)>0,h(\vsx)>0,s^2-\sigma'(\vsx)\geq0$ in the second inequality.

Also, to ensure the integral of $\zeta(\xi)$ is non-positive we should also let 
\begin{equation}\label{Hbig}
	H\geq\frac{4(s^2+\sigma'(v_b))\sigma''(v_b)}{5(s^2-\sigma'(v_b))},
\end{equation}
which is used in \eqref{interact}. We set $C_0=(s^2-\sigma'(v_b))^2\bar{w}(v_b)$, $C_b=C_N\bar{a}^\frac{3}{4}(v_b)$ so that $T(\vsx), w(\vsx) \in C[v_{-1},\avr]$. 
\end{part}
\begin{part}{Construction of $T$ and $w$ in $[\avl,0]$}
	\
	Finally, we construct  $T(\vsx)$ and $w(\vsx)$ on  $[\avl,0]$. To distinguish with the weight function and transform one on $[0,\avr]$, we denote them by $\wh(\vsx)$ and $\th(\vsx)$. Since $\sigma''(\vsx)<0$ on $\vsx\in[\avl,0]$, we define $\wh(\vsx)=\frac{C_d}{\sigma'(\vsx)},$ $\th(\vsx)=C_d^{\frac{3}{4}}$ for some constant on $C_d$. It is clear from \eqref{X-form}-\eqref{Z-form} that
	 \begin{equation}
	 	X(\vsx)=Y(\vsx)=0,\quad Z(\vsx)=-\frac{\sigma''(\vsx)}{2\sigma'(\vsx)}\geq \frac{|\sigma''(\vsx)|}{2\sigma'(\avl)},\text{ for }\vsx\in[\avl,0].
	 \end{equation} 
	We further extend the definition of $\hat{a}(\vsx)$ to the interval $[0,v_d]$ for some $v_d>0$ to be chosen such that $\wh(v_d)=0$.
	Define $\hat{a}(\vsx):=\wh(\vsx)\sigma'(\vsx)=v_d^2-(\vsx)^2$, and $\th(\vsx)=\hat{a}^{\frac{3}{4}}(\vsx)$.
	Then, plugging the formula of $\hat{a}(\vsx)$ and $\th(\vsx)$ into \eqref{X-form}-\eqref{Z-form} yields that
	\begin{equation}
		\begin{aligned}
			X(\vsx)&=-\frac{a''(\vsx)}{2}(\frac{h(\vsx)}{s\mu})^2-\frac{a'(\vsx)h(\vsx)h'(\vsx)}{2(s\mu)^2}\\
			&=\frac{h(\vsx)}{\mu ^2s^2}[h(\vsx)-\vsx(s^2-\sigma'(\vsx))],
		\end{aligned}
	\end{equation}
	\begin{equation}
		Y(\vsx)=-\frac{\hat{a}'(\vsx)}{4\hat{a}(\vsx)}=\frac{\vsx}{2(v_d^2-(\vsx)^2)}\geq 0,\text{ for }\vsx\in[0,v_d],
	\end{equation}
and
	\begin{equation}
		Z(\vsx)=s\wh(\vsx)\left(\frac{\vsx(s^2-\sigma'(\vsx))}{2s^2(v_d^2-(\vsx)^2)}-\frac{\sigma''(\vsx)}{2\sigma'(\vsx)}\right)
	\end{equation}
	Note that $X(0)>0$ and $X(\vsx)$ is continuous in $\vsx$. Then, there exists a $v_d \in (0,\avr]$ such that 
	\begin{align*}
		X (\vsx)\geqq 0, \quad \text{for}~\vsx \in [0,v_d].
	\end{align*}
For $Z(\vsx)$, it is obvious that
\begin{align*}
	Z(\vsx)\geq s\wh\left(\frac{\vsx(s^2-\sigma'(\avr))}{2s^2v_d^2}-\frac{\sigma''(\vsx)}{2\sigma'(0)}\right).
\end{align*}
	Since $\mathop{lim}\limits_{\vsx \to 0}\frac{\sigma''(\vsx)}{\vsx} = \sigma'''(0)$, we can find $v_d \in (0,\avr]$ sufficiently small such that 
	\begin{equation}\label{hatc1}
			\frac{\vsx(s^2-\sigma'(\avr))}{2s^2v_d^2}-\frac{\sigma''(\vsx)}{2\sigma'(0)}\geqq\sigma''(\vsx), \quad \text{for all}~\vsx \in [0,v_d].
	\end{equation}
	Therefore, we have $Y(\vsx)\geq0$, $Z(\vsx)\geq s\wh(\vsx)\sigma''(\vsx)$, $X(\vsx)\geqq 0$ for $\vsx \in [0,v_d]$.
	
	To ensure $w(\vsx)$ and $T(\vsx)$ being $C^1$ functions, it suffices to choose $C_d=v_d^2$. Therefore, the transform and weighted functions are defined on $[\avl, v_d]$ for sufficiently small $v_d>0$ by 
	\begin{equation}\label{wh}
		\wh(\vsx)=
		\begin{cases}
			\frac{(v_d)^2-(\vsx)^2}{\sigma'(\vsx)},\quad \vsx \in [0,v_d],\\
			\frac{(v_d)^2}{\sigma'(\vsx)},\quad \vsx \in [\avl,0],
		\end{cases}
	\end{equation}
	\begin{equation}\label{th}
		\th(\vsx)=
		\begin{cases}
			\ah^\frac{3}{4}(\vsx),\quad \vsx \in [0,v_d],\\
			\ah^\frac{3}{4}(0),\quad \vsx \in [\avl,0],
		\end{cases}
	\end{equation}
	obviously, $\wh \in C^1[\avl,v_d]$, $\th \in C^1[\avl,v_d]$.
\end{part}

\subsubsection{Energy Estimates for the transformed system}\label{construction of function2}

%As $(\vsx)'>0$, $\vsx$ is monotonely increasing, we set
%	\begin{align*}
%		&\vsx(\xi_{-1}+\X(t))=v^S(\xi_{-1})=v_{-1},\\
%		&\vsx(\xi_0+\X(t))=v^S(\xi_0)=v_0=0,\\
%		&\vsx(\xi_1+\X(t))=v^S(\xi_1)=v_1,\\
%		&\vsx(\xi_2+\X(t))=v^S(\xi_2)=v_2,\\
%		&\vdots\\
%		&\vsx(\xi_N+\X(t))=v^S(\xi_N)=v_N=v_b.
%	\end{align*}
%we get \eqref{transystem} and \eqref{transe} in these spatial parts $(-\infty,\xi_{-1}+\X(t)]$, $[\xi_{-1}+\X(t),\xi_{0}+\X(t)]$, $[\xi_{0}+\X(t),\xi_{1}+\X(t)]$, 
%$\cdots[\xi_{N-1}+\X(t),\xi_{N}+\X(t)]$, $[\xi_{N}+\X(t),+\infty)$.

To prove Lemma \ref{lowterm1}, we firstly have the energy estimate of \eqref{mainsystem} as follows.
\begin{lemma}\label{construction}
	Under the assumption of Proposition \ref{mainpro}, there exists $\varepsilon_0>0$ and $\delta_0>0$ and such that if $\varepsilon<\varepsilon_0$ and $\delta<\delta_0$, then we have the estimate as follows,
	\begin{equation}\label{transformestimate}
		\begin{aligned}
			&\mathop{sup}\limits_{t \in [0,\T]}\norm{\Phi,W}^2(t) +\int_{0}^{\T}\int_{\R} |(\vsx)'\sigma''(\vsx)|W^2\,d\xi \,dt\\
			&\leq C(\varepsilon+\norm{\Phi_0,W_0}^2+\delta\int_{0}^{T} \norm{\p_{\xi}\Phi}^2 \,dt ),
		\end{aligned}
	\end{equation}
	where $C>0$ is independent of $\varepsilon,\delta$ and $\T$.
\end{lemma}
\begin{proof}
	The estimate \eqref{transformestimate} would follows from integrating \eqref{transe} with respect to $\xi$ on $\mathbb{R}$, \eqref{XYZ-cond} and estimates for $\textbf{I}_{1,i}$,$i=1,\cdots,6$. Due to our construction of the weighted and transform functions, we will integrate separately with respect to $\xi$ on the intervals for $\xi$ corresponding to $\vsx$ in $[0,\avr]$ and $[\avl,0]$. Since $\vsx$ is an increasing function of $\xi$, let
	$$\vsx(\xi_i+\X(t))=v^S(\xi_i)=v_i, i=-1,0,1,\cdots,N$$
	$$\vsx(\xi_d+\X(t))=v^S(\xi_d)=v_d$$
	with $v_i=\frac{i}{2H},i=-1,\cdots, N-1$, $v_N=v_b$, $v_d$ defined in section 3.1.2.  
	
	First, we integrate \eqref{transe} over $[0,\T]\times[\xi_{-1}+\X(t),\xi_0+\X(t))$,  \dots, $[0,\T]\times[\xi_{N-1}+\X(t),\xi_N+\X(t)]$ and 
	$[0,\T]\times[\xi_N+\X(t),+\infty)$ and add the resulting equations together to get
\begin{equation}\label{pst-part}
	\begin{aligned}
		&\int_{\xi_{-1}+\X(t)}^{+\infty}\left(\frac{a(\vsx)}{2}\ph^2+\frac{w(\vsx)}{2}\w^2\right)(t)\,d\xi+\int_{0}^{\T}\int_{\xi_{-1}+\X(t)}^{+\infty}\mu a(\vsx)(\p_{\xi}\ph)^2\,d\xi dt\\
		&+\int_{0}^{\T}\int_{\xi_{-1}+\X(t)}^{+\infty}\mu X(\vsx)\ph^2+\frac{a(\vsx)}{s}(\vsx)'Y(\vsx)(s\ph+\w)^2+(\vsx)'Z(\vsx)\w^2\,d\xi dt\\
		&=\sum_{i=1}^{6}\int_{0}^{\T}\int_{\xi_{-1}+\X(t)}^{+\infty}\textbf{I}_{1,i}\,d\xi dt.
	\end{aligned}
\end{equation}
	It remains to control the right hand side of \eqref{pst-part}. By the construction of $T(\vsx),w(\vsx)$ for $\vsx\in[v_{-1},\avr]$, we have $T(\vsx),w(\vsx) \in C[v_{-1},\avr]$,
	but not in $C^1[v_{-1},\avr]$.  Thus,
	$$\zeta(\xi_{-1}+\X(t))=\zeta(+\infty)=0,$$
	and
	$$[\zeta](\xi_i+\X(t)):=\zeta_{-\X}(\xi_i-)-\zeta_{-\X}(\xi_i+)=\mu aY(\vsx)'\ph^2(v_i^+)-\mu aY(\vsx)'\ph^2(v_i^-),i=1,\cdots, N.$$
	It is easy to check (c.f. Appendix) that
	\begin{equation}\label{interact}
		\begin{aligned}
			&[\zeta](v_i)=0,~~i=1,2,3,\dots,N-1,\\
			&[\zeta](v_N)=\frac{\mu}{2}[-\frac{5H}{2^{\frac{13}{4}}}+\frac{(s^2+\sigma'(v_b))\sigma''(v_b)}{2^{\frac{5}{4}}(s^2-\sigma'(v_b))}]\leq 0,
		\end{aligned}
	\end{equation}
	for $H$ being chosen large enough in section 3.1.2, especially in \eqref{Hbig}. 
 Therefore, 
	\begin{equation}\label{flux-term1}
		\int_{0}^{\T}\int_{\xi_{-1}+\X(t)}^{+\infty}\textbf{I}_{1,6}\,d\xi dt=\int_{0}^{\T}\sum_{i=1}^{N}[\zeta](v_i)\leq 0
	\end{equation}
For other terms $\int_{0}^{\T}\int_{\xi_{-1}+\X(t)}^{+\infty}\textbf{I}_{1,i}\,d\xi dt,i=1,\cdots,5$,	it follows from Lemma \ref{decay} and \ref{XYdecay} that
	\begin{equation}\label{traveling-wave-decay}
		\norm{\vsx-\mathbf{\tilde{v}}}_{L^{\infty}(\R)}=\norm{(\mathbf{v}_l-\avl)(1-g_{\X})+(\mathbf{v}_r-\avr)g_{\X}}_{L^{\infty}(\R)}\leq C\varepsilon e^{-\alpha t}.
	\end{equation}
	\begin{equation}\label{X'decay}
		|\X'(t)|\leq C\varepsilon e^{-\alpha t}.
	\end{equation}
Note that $T(\vsx),w(\vsx),T'(\vsx),w'(\vsx)$ is bounded for $\vsx\in[v_{i},v_{i+1}]$, $i=-1,\cdots,N$, and$\vsx\in[v_N, \avr]$, and $T(\vsx)>0$ is bounded from below. By \eqref{X'decay}, we have 
\begin{equation}\label{X'est}
	\begin{aligned}
		\int_{0}^{\T}\int_{\xi_{-1}+\X(t)}^{+\infty}|\mathbf{I}_{1,1}|\,d\xi\,dt&\leq \int_{0}^{\T}\int_{\xi_{-1}+\X(t)}^{+\infty}|\X'(t)|(\ph^2+\w^2)\,d\xi\,dt\\
		&\leq C\varepsilon\int_{0}^{\T}e^{-\alpha t}(\norm{\ph}^2+\norm{\w}^2) \,dt \\
		&\leq C\varepsilon(\supt\norm{\ph}^2+\supt\norm{\w}^2).
	\end{aligned}
\end{equation}
	By \eqref{hterm}, one has $|\sigma(\mathbf{v}|\mathbf{\tilde{v}})|\leq C|\p_{\xi}\Phi|^2$. Then, we can get from \eqref{small1} and Lemma  \ref{decay-error} that,
	\begin{equation}\label{nondegenerate1}
		\begin{aligned}
			&\int_{0}^{\T} \int_{\xi_{-1}+\X(t)}^{+\infty}|\mathbf{I}_{1,2}|\,d\xi\,dt\leq\int_{0}^{\T} \int_{\xi_{-1}+\X(t)}^{+\infty}|(|\p_\xi\Phi|^2+|\mathbf{H}_2|+|\p_{\xi}\mathbf{H}_1|)|\w|\,d\xi\,dt\\
			&\leq C\supt\norm{W}_{L^{\infty}(\R)}\int_{0}^{\T}\norm{\p_{\xi}\Phi}^2\,dt+C\int_{0}^{\T} (\norm{\p_{\xi}\mathbf{H}_1}+\norm{\mathbf{H}_2})\norm{\w} \,dt\\
			&\leq C\delta\int_{0}^{\T}\norm{\p_{\xi}\Phi}^2\,dt+C\varepsilon\int_{0}^{\T} e^{-\alpha t}\norm{\w} \,dt\\
			&\leq C\delta\int_{0}^{\T}\norm{\p_{\xi}\Phi}^2\,dt+C\varepsilon+C\varepsilon \supt\norm{\w}^2,
		\end{aligned}
	\end{equation}
	and
	\begin{equation}\label{nondegenerate2}
		\begin{aligned}
			\int_{0}^{\T} \int_{\xi_{-1}+\X(t)}^{+\infty}|\mathbf{I}_{1,3}|\,d\xi\,dt\leq C\int_{0}^{\T}\norm{\mathbf{H}_1}\norm{\ph}  \,dt\leq C\varepsilon+C\varepsilon \supt\norm{\ph}^2.
		\end{aligned}
	\end{equation}
	By \eqref{traveling-wave-decay}, we have
	\begin{equation}\label{nondegenerate3}
		\begin{aligned}
			\int_{0}^{\T} \int_{\xi_{-1}+\X(t)}^{+\infty}|\mathbf{I}_{1,4}|\,d\xi\,dt&\leq C\int_{0}^{\T} \int_{\xi_{-1}+\X(t)}^{+\infty}|\mathbf{\tilde{v}}-\vsx||\p_{\xi}\ph||\w|\,d\xi\,dt\\
			&\leq C\varepsilon\int_{0}^{\T}e^{-\alpha t}(\norm{\p_{\xi}\ph}^2+\norm{\w}^2)  \,dt\\
			&\leq C\varepsilon\int_{0}^{\T} \norm{\p_{\xi}\ph}^2 \,dt+C\varepsilon\supt\norm{\w}^2.
		\end{aligned}
	\end{equation}
	Similarly, by \eqref{traveling-wave-decay} and \eqref{X'decay}, one has that
	\begin{equation}\label{nondegenerate4}
		\begin{aligned}
			\int_{0}^{\T} \int_{\xi_{-1}+\X(t)}^{+\infty}|\mathbf{I}_{1,5}|\,d\xi\,dt&\leq C\int_{0}^{\T} \int_{\xi_{-1}+\X(t)}^{+\infty}|\mathbf{\tilde{v}}-\vsx||\ph||\w|\,d\xi\,dt\\
			&\leq C\varepsilon\int_{0}^{\T} e^{-\alpha t}(\norm{\ph}^2+\norm{\w}^2) \,dt\\
			&\leq C\varepsilon(\supt\norm{\ph}^2+\supt\norm{\w}^2).
		\end{aligned}
	\end{equation}
	Plugging \eqref{flux-term1},\eqref{X'est}-\eqref{nondegenerate4} into \eqref{pst-part} and using \eqref{XYZ-cond} for $\vsx\in[v_{-1},\avr]$, we have
	\begin{equation}\label{pst-integral}
		\begin{aligned}
			&\int_{\xi_{0}+\X(t)}^{+\infty}\left(\frac{a(\vsx)}{2}\ph^2+\frac{w(\vsx)}{2}\w^2\right)(t)\,d\xi+\int_{0}^{\T}\int_{\xi_{0}+\X(t)}^{+\infty}(\vsx)'|\sigma''(\vsx)|\w^2\,d\xi dt\\
			&\leq\int_{\xi_{-1}+\X(t)}^{+\infty}\left(\frac{a(\vsx)}{2}\ph^2+\frac{w(\vsx)}{2}\w^2\right)(t)\,d\xi+\int_{0}^{\T}\int_{\xi_{-1}+\X(t)}^{+\infty}(\vsx)'|\sigma''(\vsx)|\w^2\,d\xi dt\\
			&\leq C\varepsilon(\supt\norm{\ph}^2+\supt\norm{\w}^2)+C\varepsilon+C\delta\int_{0}^{\T}\norm{\p_{\xi}\Phi}^2\,dt
		\end{aligned}
	\end{equation}
	where the first inequality follows from the non-negative of integrands.
	
	 Next, we integrate \eqref{transe} $[0,\T]\times[-\infty,\xi_0+\X(t))$ and $[0,\T]\times[\xi_0+\X(t),\xi_d+\X(t))$, and add the resulting equations together to yield that
	 \begin{equation}\label{neg-part}
	 	\begin{aligned}
	 		&\int_{-\infty}^{\xi_d+\X(t)}\left(\frac{\hat{a}(\vsx)}{2}\ph^2+\frac{\hat{w}(\vsx)}{2}\w^2\right)(t)\,d\xi+\int_{0}^{\T}\int_{\xi_{-1}+\X(t)}^{+\infty}\mu \hat{a}(\vsx)(\p_{\xi}\ph)^2\,d\xi dt\\
	 		&+\int_{0}^{\T}\int_{-\infty}^{\xi_d+\X(t)}\mu X(\vsx)\ph^2+\frac{a(\vsx)}{s}(\vsx)'Y(\vsx)(s\ph+\w)^2+(\vsx)'Z(\vsx)\w^2\,d\xi dt\\
	 		&=\sum_{i=1}^{6}\int_{0}^{\T}\int_{-\infty}^{\xi_d+\X(t)}\textbf{I}_{1,i}\,d\xi dt.
	 	\end{aligned}
	 \end{equation}
	
%	Then we estimate the rest part, set $w(\vsx)=\wh(\vsx)$, $T(\vsx)=\th(\vsx)$, then $(\Phi,W)(\xi,t)=\th(\vsx(\xi))(\phh(\xi,t),\whh(\xi,t))$, substisubstituting these into \eqref{transe}, we have
%	\begin{equation}\label{transe1}
%		\begin{aligned}
%			&\p_t(\frac{\ah(\vsx)}{2}\phh^2+\frac{\wh(\vsx)}{2}\whh^2)+\mu(\p_{\xi}\phh)^2\ah(\vsx)+\whh^2(\vsx)'Z(\vsx)\\
%			&+Y(\vsx)(\vsx)'(s\phh+\whh)^2+\mu\phh^2X(\vsx)\\
%			&=\underbrace{\frac{\X'(\vsx)'(\phh^2(2\ah(\vsx)\th'(\vsx)-\ah'(\vsx))+\whh^2(2\wh(\vsx)\th'(\vsx)-\wh'(\vsx)))}{2\th(\vsx)}}_{\mathbf{I}_{1,7}}\\
%			&+\underbrace{\frac{\sigma(\mathbf{v}|\mathbf{\tilde{v}})\wh(\vsx)\whh+(\mathbf{H}_2-\mu\p_{\xi}\mathbf{H}_1)\wh(\vsx)\whh}{\th(\vsx)}}_{\mathbf{I}_{1,8}}+\underbrace{\frac{\mathbf{H}_1\phh \ah(\vsx)}{\th(\vsx)}}_{\mathbf{I}_{1,9}}+\underbrace{(\sigma'(\mathbf{\tilde{v}})-\sigma'(\vsx))\wh(\vsx)\p_{\xi}\phh\whh}_{\mathbf{I}_{1,10}}\\
%			&+\underbrace{\frac{\th'(\vsx)}{\th(\vsx)}(\sigma'(\mathbf{\tilde{v}})-\sigma'(\vsx))\wh(\vsx)(\vsx)'\phh\whh}_{\mathbf{I}_{1,11}}+\underbrace{\p_{\xi}\left\{\dots\right\}}_{\mathbf{I_{1,12}}},\\
%		\end{aligned}
%	\end{equation}
%	where $\left\{\dots\right\}=\frac{s \ah(\vsx)}{2}\phh^2+\frac{s\wh(\vsx)}{2}\whh^2+\whh\phh \ah(\vsx)+\mu\frac{\th'(\vsx)}{\th(\vsx)}\ah(\vsx)(\vsx)'\phh^2-\frac{\mu}{2}\ah'(\vsx)(\vsx)'\phh^2+\mu \ah(\vsx)\p_{\xi}\phh\phh$.
	Note that $T(\vsx),w(\vsx) \in C^1[\avl,v_d]$, it is easy to check (c.f. Appendix) that
	\begin{equation}\label{interact2}
		\zeta(-\infty)=0,\quad \zeta(\xi_d+\X(t))\leq 0.
	\end{equation}
	Therefore, 
	\begin{equation}\label{flux-term2}
		\int_{0}^{\T}\int_{-\infty}^{\xi_{d}+\X(t)}\textbf{I}_{1,6}\,d\xi dt\leq 0.
	\end{equation}
%	 Integrate $\mathbf{I_{1,12}}$ over $[0,\T] \times (-\infty,\xi_0+\X(t)]$ and $[0,\T] \times [\xi_0+\X(t),\xi_d+\X(t)]$ we find it is non-positive, the details are also left in "Appendix".
%	
%	Set $\vsx(\xi_d+\X(t))=v_d$, we find $\wh(\vsx),\th(\vsx) \in C^1[\avl,v_d]$ but $\th''(\vsx)$ is not continuous at the point $\xi_0+\X(t)$, and we have
%\begin{align}
%	&X(\vsx):=\p_{\xi}\th\p_{\xi}\left(\frac{\ah}{\th}\right)-\frac{1}{2}\p_{\xi}^2\ah\nonumber\\
%	&=\big[\frac{\th'(\vsx)}{\th(\vsx)}\ah'(\vsx)-\big(\frac{\th'(\vsx)}{\th(\vsx)}\big)^2\ah(\vsx)-\frac{\ah''(\vsx)}{2}\big]((\vsx)')^2-\frac{\ah'(\vsx)(\vsx)''}{2}\nonumber\\
%	&=\big[\frac{\th'(\vsx)}{\th(\vsx)}\ah'(\vsx)-\big(\frac{\th'(\vsx)}{\th(\vsx)}\big)^2\ah(\vsx)-\frac{\ah''(\vsx)}{2}\big](\frac{h(\vsx)}{s\mu})^2-\frac{\ah'(\vsx)h(\vsx)h'(\vsx)}{2(s\mu)^2},
%\end{align}
%\begin{equation}
%	Y(\vsx):=\frac{\ah'(\vsx)}{2\ah(\vsx)}-\frac{\th'(\vsx)}{\th(\vsx)},
%\end{equation}
%\begin{align}
%	Z(\vsx):=\frac{s\wh'(\vsx)}{2}-s\frac{\th'(\vsx)}{\th(\vsx)}w(\vsx)+\frac{\th'(\vsx)}{s\th(\vsx)}\ah(\vsx)-\frac{\ah'(\vsx)}{2s}.
%\end{align}

Similar to \eqref{X'est}-\eqref{nondegenerate4}, we have the following estimates for other terms on the right hand side of \eqref{neg-part}
	\begin{equation}\label{hX'est}
		\begin{aligned}
			\int_{0}^{\T}\int_{-\infty}^{\xi_{d}+\X(t)}|\mathbf{I}_{1,1}|\,d\xi\,dt\leq C\varepsilon(\supt\norm{\ph}^2+\supt\norm{\w}^2),
		\end{aligned}
	\end{equation}
	\begin{equation}\label{hnondegenerate1}
			\int_{0}^{\T} \int_{-\infty}^{\xi_{d}+\X(t)}|\mathbf{I}_{1,2}|\,d\xi\,dt\leq C\delta\int_{0}^{\T}\norm{\p_{\xi}\Phi}^2\,dt+C\varepsilon+C\varepsilon \supt\norm{\w}^2,
	\end{equation}
	\begin{equation}\label{hnondegenerate2}
		\begin{aligned}
			\int_{0}^{\T} \int_{-\infty}^{\xi_{d}+\X(t)}|\mathbf{I}_{1,3}|\,d\xi\,dt\leq C\varepsilon+C\varepsilon \supt\norm{\ph}^2,
		\end{aligned}
	\end{equation}
	\begin{equation}\label{hnondegenerate3}
		\begin{aligned}
			\int_{0}^{\T} \int_{-\infty}^{\xi_{d}+\X(t)}|\mathbf{I}_{1,4}|\,d\xi\,dt
			\leq C\varepsilon\int_{0}^{\T} \norm{\p_{\xi}\ph}^2 \,dt+C\varepsilon\supt\norm{\w}^2,
		\end{aligned}
	\end{equation}
	\begin{equation}\label{hnondegenerate4}
		\begin{aligned}
			\int_{0}^{\T} \int_{-\infty}^{\xi_{d}+\X(t)}|\mathbf{I}_{1,5}|\,d\xi\,dt\leq C\varepsilon(\supt\norm{\ph}^2+\supt\norm{\w}^2).
		\end{aligned}
	\end{equation}
	Plugging \eqref{flux-term2}-\eqref{hnondegenerate4} into \eqref{neg-part} and using \eqref{XYZ-cond} for $\vsx\in[\avl,v_d]$, we have
	\begin{equation}\label{neg-integral}
		\begin{aligned}
			&\int_{-\infty}^{\xi_{0}+\X(t)}\left(\frac{\hat{a}(\vsx)}{2}\ph^2+\frac{\hat{w}(\vsx)}{2}\w^2\right)(t)\,d\xi+\int_{-\infty}^{\xi_{d}+\X(t)}(\vsx)'|\sigma''(\vsx)|\w^2\,d\xi dt\\
			&\leq\int_{-\infty}^{\xi_{0}+\X(t)}\left(\frac{\hat{a}(\vsx)}{2}\ph^2+\frac{\hat{w}(\vsx)}{2}\w^2\right)(t)\,d\xi+\int_{-\infty}^{\xi_{d}+\X(t)}(\vsx)'|\sigma''(\vsx)|\w^2\,d\xi dt\\
			&\leq C\varepsilon(\supt\norm{\ph}^2+\supt\norm{\w}^2)+C\varepsilon+C\delta\int_{0}^{\T}\norm{\p_{\xi}\Phi}^2\,dt
		\end{aligned}
	\end{equation}
	Note that $w(\vsx)$ and $\wh(\vsx)$ have uniform lower bound $min\big\{\frac{1}{2H},\frac{(v_d)^2}{\sigma'(0)}\big\}$ whenever $\vsx \in [0,\avr]$ and $\vsx \in [\avl,0]$, and $T$ is also bounded from below and above so that $\norm{\ph,\w} \sim \norm{\Phi,W}$.  Summing \eqref{pst-integral} with \eqref{neg-integral} and using the lower bounds of $w$ and $a$, we can get \eqref{transformestimate}.
%	
%	
%	If $\varepsilon_0$ is small enough, integrate \eqref{transe} over $[0,\T]\times[\xi_{-1}+\X(t),\xi_0+\X(t)]$, $[0,\T]\times[\xi_0+\X(t),\xi_1+\X(t)]$, \dots, $[0,\T]\times[\xi_{N-1}+\X(t),\xi_N+\X(t)]$ and 
%	$[0,\T]\times[\xi_N+\X(t),+\infty)$, integrate \eqref{transe1} over $[0,\T] \times (-\infty,\xi_0+\X(t)]$ and $[0,\T] \times [\xi_0+\X(t),\xi_d+\X(t)]$, then adding them, consider \eqref{interact}, and $\norm{\ph,\w} \sim \norm{\Phi,W} \sim \norm{\phh,\whh}$ by the boundedness of $T(\vsx)$ and $\th(\vsx)$, 
%	$w(\vsx)\sim T(\vsx) \sim \wh(\vsx) \sim \th(\vsx) \sim constant$.
%	
%	We also consider that when $\vsx \in [0,+\infty)$ and $\vsx \in (-\infty,0]$, $w(\vsx)$ and $\wh(\vsx)$ have uniform lower bound $min\big\{\frac{1}{2H},\frac{(v_d)^2}{\sigma'(0)}\big\}$, collecting the estimates from
%	$\mathbf{I}_{1,1}$ to $\mathbf{I}_{1,12}$, we finally have \eqref{transformestimate},
%	where $C$ is independent of $\varepsilon$, $\delta$ and $\T$.
\end{proof}
\subsubsection{Proof of Lemma \ref{lowterm1}}\label{pflow1}
To prove Lemma \ref{lowterm1}, it suffices to control $\int_{0}^{\T}\|\p_\xi\Phi\|^2\,d\xi dt$. Multiplying \subeqref{mainsystem}{1} by $\Phi$ and  \subeqref{mainsystem}{2} by $\frac{W}{\sigma'(\vsx)}$, and summing together,
we have
\begin{equation}\label{lowlemma}
	\begin{aligned}
		&\p_t\left(\frac{\Phi^2}{2}+\frac{W^2}{2\sigma'(\vsx)}\right)+\mu(\p_{\xi}\Phi)^2\\
		&=\underbrace{\frac{\X'(\vsx)'\sigma''(\vsx)}{2(\sigma'(\vsx))^2}W^2}_{\mathbf{I}_{2,1}}+\underbrace{\frac{\sigma'(\mathbf{\tilde{v}})-\sigma'(\vsx)}{\sigma'(\vsx)}\p_{\xi}\Phi W+\frac{\sigma(\mathbf{v}|\mathbf{\tilde{v}})}{\sigma'(\vsx)W}}_{\mathbf{I}_{2,2}}\\
		&\quad+\underbrace{\frac{\mathbf{H}_2-\mu\p_{\xi}\mathbf{H}_1}{\sigma'(\vsx)}W+\mathbf{H}_1\Phi}_{\mathbf{I}_{2,3}}+\underbrace{\frac{s\sigma''(\vsx)(\vsx)'}{2(\sigma'(\vsx))^2}W^2}_{\mathbf{I}_{2,4}}+\p_{\xi}\left\{\dots\right\},
	\end{aligned}
\end{equation}
where $\left\{\dots\right\}=\frac{s\Phi^2}{2}+W\Phi+\frac{sW^2}{2\sigma'(\vsx)}+\mu\p_{\xi}\Phi\Phi$, which vanishes at far fields $\xi=\pm\infty$. By \eqref{X'decay}, we have
\begin{equation}\label{low1}
	\int_{0}^{\T} \int_{\R} |{\mathbf{I}_{2,1}}| \,d\xi\,dt \leq C\varepsilon\int_{0}^{\T} e^{-\alpha t}\norm{W}^2 \,dt\leq C\varepsilon\supt\norm{W}^2.
\end{equation}
It follows from \eqref{traveling-wave-decay},  $|\sigma(\mathbf{v}|\mathbf{\tilde{v}})|\leq C|\p_{\xi}\Phi|^2$, and \eqref{small1} that
\begin{equation}\label{I2.1}
	\begin{aligned}
		\int_{0}^{\T} \int_{\R}|\mathbf{I}_{2,2}| \,d\xi  \,dt&\leq C\varepsilon\int_{0}^{\T}\int_{\R} e^{-\alpha t}|\p_{\xi}\Phi W| \,d\xi\,dt +C\delta\int_{0}^{\T}\norm{\p_{\xi}\Phi}^2\,dt\\
		&\leq C(\varepsilon+\delta)\int_{0}^{\T}\norm{\p_{\xi}\Phi}^2\,dt+C\varepsilon\supt\norm{W}^2.
	\end{aligned}
\end{equation}
We can obtain from using Lemma \ref{decay-error} that
\begin{equation}\label{I2.2}
	\begin{aligned}
		\int_{0}^{\T} \int_{\R}|\mathbf{I}_{2,3}|  \,d\xi  \,dt&\leq C\int_{0}^{\T}(\norm{\mathbf{H}_2}+\norm{\p_{\xi}\mathbf{H}_1})\norm{W}+\norm{\mathbf{H}_1}\norm{\Phi}  \,dt \\
		&\leq C\varepsilon+C\varepsilon(\supt\norm{W}^2+\supt\norm{\Phi}^2).
	\end{aligned}
\end{equation}
Moreover, by lemma \ref{construction}, we have
\begin{equation}\label{I2.3}
	\begin{aligned}
		&\int_{0}^{\T} \int_{\R}|\mathbf{I}_{2,4}| \,d\xi  \,dt \leq C(\varepsilon+\norm{\Phi_0,W_0}^2+\delta\int_{0}^{\T} \norm{\p_{\xi}\Phi}^2 \,dt ).
	\end{aligned}
\end{equation}
 Integrating \eqref{lowlemma} over $\R\times[0,\T]$, plugging in the above estimates, and choosing $\varepsilon_0,\delta_0$ small enough, we can 
prove Lemma \ref{lowterm1}.
\subsection{Higher-order Energy Estimates}\label{3.2sec}
Now we start to prove the higher-order estimates of solution $(\Phi,W)$ to \eqref{mainsystem}. Due to the non-convexity of constitutive relation,  we still can not directly
use the energy estimate method as in \cite{HuangYuanCMP} to control the first order energy.  However, Lemma \ref{lowterm1} allows us to apply the weighted energy estimates method via a newly constructed weight function that guarantees the term involving $(\partial_{\xi}W)^2$ is positive. The following lemma gives the first order energy estimates of $(\Phi, W)$
\begin{lemma}\label{lowterm2}
	Under the assumptions of Proposition \ref{mainpro}, there exist $\varepsilon_0>0$ and $\delta_0>0$ such that if $\varepsilon<\varepsilon_0$ and $\delta<\delta_0$, then
	\begin{equation}\label{estimate2}
		\supt\norm{\p_{\xi}(\Phi,W)}^2(t)+\int_{0}^{\T}\norm{\p^2_{\xi}\Phi}^2 \,dt\leq C\varepsilon+C\norm{(\Phi_0,W_0)}^2_1,
	\end{equation}
	where $C>0$ is independent of $\varepsilon,\delta$ and $\T$.
\end{lemma}
\begin{proof}
	Let $\tilde{w}(\vsx)$ be a weight function defined by
	\begin{equation}\label{newweight}
		\tilde{w}(\vsx)=
		\begin{cases}
			\frac{\sigma'(\vsx)}{(\sigma'(0))^2},\quad \vsx \in [0,\avr]\\
			\frac{1}{\sigma'(\vsx)},\quad \vsx \in [\avl,0],
		\end{cases}
	\end{equation}
and let $\at(\vsx):=\wt(\vsx)\sigma'(\vsx)$. In view of \eqref{SC}-\eqref{SC2}, it is easy to show that
\begin{equation}
	\tilde{w}'(\vsx) \geq 0,\quad \tilde{a}'(\vsx) \geq 0, \text{ on } [\avl,\avr].
\end{equation}
Multiplying \subeqref{mainsystem}{1}, \subeqref{mainsystem}{2} by $-\at(\vsx)\p^2_{\xi}\Phi$ and $-\wt(\vsx)\p^2_{\xi}W$ respectively, and summing the resulting equations, we have
	\begin{equation}\label{relt2}
		\begin{aligned}
			&\p_t\left(\frac{\at(\vsx)}{2}(\p_{\xi}\Phi)^2+\frac{\wt(\vsx)}{2}(\p_{\xi}W)^2\right)+\mu\at(\vsx)(\p^2_{\xi}\Phi)^2\\
			&+\frac{s(\vsx)'\at'(\vsx)}{2}(\p_{\xi}\Phi)^2+\frac{s(\vsx)'\wt'(\vsx)}{2}(\p_{\xi}W)^2\\
			&=-\sigma'(\vsx)\wt'(\vsx)(\vsx)'\p_{\xi}\Phi\p_{\xi}W-\mu\at'(\vsx)(\vsx)'\p^2_{\xi}\Phi\p_{\xi}\Phi\\
			&+\underbrace{(\p_{\xi}\mathbf{H}_2-\mu\p^2_{\xi}\mathbf{H}_1)\wt(\vsx)\p_{\xi}W-\mathbf{H}_1\at'(\vsx)(\vsx)'\p_{\xi}\Phi-\mathbf{H}_1\at(\vsx)\p^2_{\xi}\Phi}_{\mathbf{I}_{3,1}}\\
			&+\underbrace{\p_{\xi}\sigma(\mathbf{v}|\mathbf{\tilde{v}})\wt(\vsx)\p_{\xi}W+(\sigma'(\mathbf{\tilde{v}})-\sigma'(\vsx))\wt(\vsx)\p^2_{\xi}\Phi\p_{\xi}W}_{\mathbf{I}_{3,2}}\\
			&+\underbrace{(\p_{\xi}\mathbf{\tilde{v}}-(\vsx)')\sigma''(\vsx)\wt(\vsx)\p_{\xi}W \p_{\xi}\Phi+(\sigma''(\mathbf{\tilde{v}})-\sigma''(\vsx))\wt(\vsx)(\vsx)'\p_{\xi}W\p_{\xi}\Phi,}_{\mathbf{I}_{3,3}}\\
			&+\underbrace{\frac{\X'(\vsx)'[(\p_{\xi}\Phi)^2\at'(\vsx)+\wt'(\vsx)(\p_{\xi}W)^2]}{2}}_{\mathbf{I}_{3,4}}+\p_{\xi}\left\{\dots\right\}
		\end{aligned}
	\end{equation}
	where
	\begin{align*}
		\left\{\dots\right\}=&\at(\vsx)\p_t\Phi\p_{\xi}\Phi-\frac{s\at(\vsx)(\p_{\xi}\Phi)^2}{2}-\at(\vsx)\p_{\xi}\Phi\p_{\xi}W+\wt(\vsx)\p_tW\p_{\xi}W\\&-\frac{s\wt(\vsx)(\p_{\xi}W)^2}{2}-(\sigma'(\mathbf{\tilde{v}})-\sigma'(\vsx))\wt(\vsx)\p_{\xi}W\p_{\xi}\Phi-\sigma(\mathbf{v}|\mathbf{\tilde{v}})\wt(\vsx)\p_{\xi}W\\
		&-(\mathbf{H}_2-\mu\p_{\xi}\mathbf{H}_1)\wt(\vsx)\p_{\xi}W,
	\end{align*}
which vanishes at far fields $\xi=\pm\infty$.
	By using Young inequality, one has
	\begin{equation}\label{cauchyine}
		\begin{aligned}
			&|\wt'(\vsx)\sigma'(\vsx)(\vsx)'\p_{\xi}\Phi\p_{\xi}W|\leq \frac{s\wt'(\vsx)(\vsx)'(\p_{\xi}W)^2}{4}+\frac{\wt'(\vsx)(\sigma'(\vsx))^2(\vsx)'(\p_{\xi}\Phi)^2}{s},\\
			&|\mu\at'(\vsx)(\vsx)'\p^2_{\xi}\Phi\p_{\xi}\Phi|\leq \frac{\mu\at(\vsx)}{4}(\p^2_{\xi}\Phi)^2+\frac{\mu(\at'(\vsx))^2((\vsx)')^2}{\at(\vsx)}(\p_{\xi}\Phi)^2,
		\end{aligned}
	\end{equation}
	Plugging \eqref{cauchyine} into \eqref{relt2}, we have
	\begin{equation}\label{relt3}
		\begin{aligned}
			&\p_t\left(\frac{\at(\vsx)}{2}(\p_{\xi}\Phi)^2+\frac{\wt(\vsx)}{2}(\p_{\xi}W)^2\right)+\frac{3\mu}{4}\at(\vsx)(\p^2_{\xi}\Phi)^2\\
			&+\frac{s(\vsx)'\at'(\vsx)}{2}(\p_{\xi}\Phi)^2+\frac{s(\vsx)'\wt'(\vsx)}{4}(\p_{\xi}W)^2\\
			&\leq\p_{\xi}\left\{\dots\right\}+\mathbf{I}_{3,1}+\mathbf{I}_{3,2}+\mathbf{I}_{3,3}+\mathbf{I}_{3,4}\\
%			&+\underbrace{(\p_{\xi}\mathbf{H}_2-\mu\p^2_{\xi}\mathbf{H}_1)\wt(\vsx)\p_{\xi}W-\mathbf{H}_1\at'(\vsx)(\vsx)'\p_{\xi}\Phi-\mathbf{H}_1\at(\vsx)\p^2_{\xi}\Phi}_{\mathbf{I}_{3,1}}\\
%			&+\underbrace{\p_{\xi}\sigma(\mathbf{v}|\mathbf{\tilde{v}})\wt(\vsx)\p_{\xi}W+(\sigma'(\mathbf{\tilde{v}})-\sigma'(\vsx))\wt(\vsx)\p^2_{\xi}\Phi\p_{\xi}W}_{\mathbf{I}_{3,2}}\\
%			&+\underbrace{(\p_{\xi}\mathbf{\tilde{v}}-(\vsx)')\sigma''(\vsx)\wt(\vsx)\p_{\xi}W \p_{\xi}\Phi+(\sigma''(\mathbf{\tilde{v}})-\sigma''(\vsx))\wt(\vsx)(\vsx)'\p_{\xi}W\p_{\xi}\Phi}_{\mathbf{I}_{3,3}},\\
			&+\underbrace{\left(\frac{\wt'(\vsx)(\sigma'(\vsx))^2(\vsx)'}{s}+\frac{\mu(\at'(\vsx))^2((\vsx)')^2}{\at(\vsx)}\right)(\p_{\xi}\Phi)^2}_{\mathbf{I}_{3,5}}.
		\end{aligned}
	\end{equation}
	Since $\mathbf{\tilde{v}}(\xi,t)=\mathbf{v}_l(\xi,t)(1-g_{\X}(\xi))+\mathbf{v}_rg_{\X}(\xi)$, it follows from Lemma \ref{decay} and \ref{XYdecay} that
		\begin{align}\label{vsx'decay}
			\norm{\p_{\xi}\mathbf{\tilde{v}}-(\vsx)'}_{L^{\infty}(\R)}&\leq C(\norm{\p_{\xi}\mathbf{v}_l}_{L^{\infty}(\R)}+\norm{\p_{\xi}\mathbf{v}_r}_{L^{\infty}(\R)}+\norm{\mathbf{v}_r-\avr}_{L^{\infty}(\R)}+\norm{\mathbf{v}_l-\avl}_{L^{\infty}(\R)})\nonumber\\
			&\leq C\varepsilon e^{-\alpha t}.
		\end{align}
	It follows from Lemma \ref{decay-error} that
		\begin{align}\label{high1.1}
		\int_{0}^{\T}\int_{\R}|\mathbf{I}_{3,1}|\,d\xi\,dt&\leq C\int_{0}^{\T} (\norm{\p_{\xi}\mathbf{H}_2}+\norm{\p^2_{\xi}\mathbf{H}_1})\norm{\p_{\xi}W}+(\norm{\p_{\xi}\Phi}+\norm{\p^2_{\xi}\Phi})\norm{\mathbf{H}_1} \,dt\nonumber\\
		&\leq C\varepsilon\int_{0}^{\T}e^{-\alpha t} (\norm{\p_{\xi}W}+\norm{\p_{\xi}\Phi}+\norm{\p^2_{\xi}\Phi}) \,dt\\
		&\leq C\varepsilon(\supt\norm{\p_{\xi}W,\p_{\xi}\Phi})+C\varepsilon\int_{0}^{T}\norm{\p^2_{\xi}\Phi}^2\,dt.\nonumber
		\end{align}
	By \eqref{hterm}, one has that $|\p_{\xi}\sigma(\mathbf{v}|\mathbf{\tilde{v}})|\leq C(\abs{\p_{\xi}\Phi}^2+\abs{\p_{\xi}\Phi}\abs{\p^2_{\xi}\Phi})$. Thus, it follows from \eqref{small1}, \eqref{traveling-wave-decay} and 
	Sobolev inequality that
		\begin{align}\label{high1.2}
			&\int_{0}^{\T}\int_{\R}|\mathbf{I}_{3,2}|\,d\xi\,dt\leq C\int_{0}^{\T}\int_{\R}(\abs{\p_{\xi}\Phi}^2+\abs{\p_{\xi}\Phi}\abs{\p^2_{\xi}\Phi})|\p_{\xi}W|+|\mathbf{\tilde{v}}-\vsx||\p^2_{\xi}\Phi\p_{\xi}W|\,d\xi\,dt\nonumber\\
			&\leq C\int_{0}^{\T}\int_{\R}\norm{\p_{\xi}{\Phi}}_{L^{\infty}}(\abs{\p_{\xi}\Phi}+\abs{\p^2_{\xi}\Phi})\abs{\p_{\xi}W}  \,d\xi\,dt+C \varepsilon\int_{0}^{\T} e^{-\alpha t}(\norm{\p^2_{\xi}\Phi}^2+\norm{\p_{\xi}W}^2)\,dt\nonumber\\
			&\leq C\supt\norm{\p_{\xi}W}\int_{0}^{\T} \norm{\p_{\xi}\Phi}^{\frac{1}{2}}\norm{\p^2_{\xi}\Phi}^{\frac{1}{2}}(\norm{\p_{\xi}\Phi}+\norm{\p^2_{\xi}\Phi})\,dt\\
			&\quad+C\varepsilon\int_{0}^{\T}\norm{\p^2_{\xi}\Phi}^2\,dt+C\varepsilon\supt\norm{\p_{\xi}W}^2\nonumber\\
			&\leq C(\delta+\varepsilon)\int_{0}^{\T}\norm{\p^2_{\xi}\Phi}^2\,dt+C\varepsilon\supt\norm{\p_{\xi}W}^2+C\int_{0}^{\T}\norm{\p_{\xi}\Phi}^2\,dt.\nonumber
		\end{align}
	By \eqref{traveling-wave-decay}, \eqref{vsx'decay}, we have 
	\begin{equation}\label{high1.3}
			\int_{0}^{\T}\int_{\R}|\mathbf{I}_{3,3}|\,d\xi\,dt\leq C\varepsilon\int_{0}^{\T}e^{-\alpha t}(\norm{\p_{\xi}W}^2+\norm{\p_{\xi}\Phi}^2)\,dt\leq C\varepsilon\supt\norm{(\p_{\xi}\Phi,\p_{\xi}W)}^2.
	\end{equation}
	Similar to \eqref{low1}, it follows from \eqref{X'decay} that
	\begin{equation}\label{low2}
			\int_{0}^{\T}\int_{\R}|\mathbf{I}_{3,4}|\,d\xi \,dt\leq C\varepsilon\int_{0}^{\T}e^{-\alpha t} (\norm{\p_{\xi}\Phi}^2+\norm{\p_{\xi}W}^2)\,dt \leq C\varepsilon\supt\norm{(\p_{\xi}\Phi,\p_{\xi}W)}^2,
	\end{equation}
	By using Lemma \ref{lowterm1}, we have
	\begin{equation}\label{P2}
			\int_{0}^{\T}\int_{\R} \abs{\mathbf{I}_{3,5}}\,d\xi\,dt\leq C\int_{0}^{\T}\norm{\p_{\xi}\Phi}^2\,dt \leq C(\norm{(\Phi_0,W_0)}^2+\varepsilon).
	\end{equation}
	Integrating \eqref{relt3} over $\R \times [0,\T]$,  one can get
	\eqref{estimate2} by substituting \eqref{high1.1}-\eqref{P2}, and choosing $\varepsilon_0$ and $\delta_0$ sufficiently small.
\end{proof}

Now we give the estimates $\p^2_{\xi}\Phi$ and the derivatives of $\Psi$.
\begin{lemma}\label{lowterm3}
	Under the assumptions of Proposition \ref{mainpro}, there exist $\varepsilon_0$ and $\delta_0$ such that if $\varepsilon<\varepsilon_0$ and $\delta<\delta_0$, then
	\begin{align*}
		\supt(\norm{\p^2_{\xi}\Phi}^2+\norm{\Psi}^2_2)(t)+\int_{0}^{\T} \norm{\p_{\xi}\Psi}^2_2(t)\,dt\leq C(\norm{(\Phi_0,\Psi_0)}^2_2+\varepsilon),
	\end{align*}
	where $C>0$ is independent of $\varepsilon,\delta$ and $\T$.
\end{lemma}
\begin{proof}
	First, by multiplying $\p_{\xi}W$ on \subeqref{mainsystem}{1} and substituting \subeqref{mainsystem}{2} into the resulting equation, we have
		\begin{equation}\label{relt4}
			\begin{aligned}
				(\p_{\xi}W)^2=&\underbrace{\p_t(\Phi\p_{\xi}W)}_{\mathbf{I}_{4,1}}-\p_{\xi}\left\{\dots\right\}+\underbrace{(\sigma(\mathbf{v})-\sigma(\mathbf{\tilde{v}}))\p_{\xi}\Phi-\mu\p^2_{\xi}\Phi\p_{\xi}W}_{\mathbf{I}_{4,2}}\\
				&-\underbrace{(\p_{\xi}\mathbf{H}_2-\mu\p^2_{\xi}\mathbf{H}_1)\Phi+\mathbf{H}_1\p_{\xi}W}_{\mathbf{I}_{4,3}},
			\end{aligned}
		\end{equation}
		where $\left\{\dots\right\}=(\sigma(\mathbf{v})-\sigma(\mathbf{\tilde{v}}))\Phi+s\p_{\xi}W\Phi$, which vanishes at far fields $\xi=\pm\infty$. 
		It is obvious that
		\begin{align*}
			\left|\int_{0}^{\T}\int_{\R} \mathbf{I}_{4,1}\,d\xi\,dt\right|\leq C(\supt\norm{\Phi}^2+\supt\norm{\p_{\xi}W}^2+\norm{\Phi_0}^2+\norm{\p_{\xi}W_0}^2),
		\end{align*}
		\begin{align*}
			\int_{0}^{\T}\int_{\R} |\mathbf{I}_{4,2}|\,d\xi\,dt&=\int_{0}^{\T}\int_{\R} |(\sigma(\mathbf{v})-\sigma(\mathbf{\tilde{v}}))\p_{\xi}\Phi-\mu\p^2_{\xi}\Phi\p_{\xi}W|\,d\xi\,dt\\
			&\leq \frac{1}{2}\int_{0}^{\T}\|\p_\xi W\|^2\,dt+C\int_{0}^{\T} \norm{\p_{\xi}\Phi}^2+\norm{\p_\xi^2\Phi}^2\,dt+C\varepsilon.
		\end{align*}
		For $\mathbf{I}_{4,3}$, we use Lemma \ref{decay-error} to get that
		\begin{align*}
			\int_{0}^{\T}\int_{\R} |\mathbf{I}_{4,3}|\,d\xi\,dt
			&\leq \int_{0}^{\T} (\norm{\p_{\xi}\mathbf{H}_2}+\norm{\p^2_{\xi}\mathbf{H}_1})\norm{\Phi}+\norm{\mathbf{H}_1}\norm{\p_{\xi}W}\,dt\\ 
			&\leq C\varepsilon\int_{0}^{\T}e^{-\alpha t} (\norm{\Phi}+\norm{\p_{\xi}W})\,dt\\
			&\leq C\varepsilon(\supt\norm{\Phi}^2+\supt\norm{\p_{\xi}W}^2+1),
		\end{align*}
		Integrating \eqref{relt4} over $[0,\T]\times\mathbb{R}$ and combining with the estimate of $\mathbf{I}_{4,1}$-$\mathbf{I}_{4,3}$ , then using Lemma \ref{lowterm1}, \ref{lowterm2}, we have
		\begin{equation}\label{W1}
			\int_{0}^{\T} \norm{\p_{\xi}W}^2 \,dt\leq C(\norm{(\Phi_0,W_0)}^2_1+\varepsilon).
		\end{equation}
		Then, we can get from \eqref{W1}, \eqref{w}, Lemma \ref{lowterm1} and \ref{lowterm2} that
		\begin{equation}\label{psi1}
			\int_{0}^{\T} \norm{\p_{\xi}\Psi}^2 \,dt\leq C(\norm{\Phi_0}^2_2+\norm{\Psi_0}^2_1+\varepsilon).
		\end{equation}
		
		Next, differentiating \subeqref{system2}{2} with respect to $\xi$ and then multiplying $\p_{\xi}\Psi$ on the resulting equation, yield that
		\begin{equation}\label{relt5}
			\p_t\left(\frac{(\p_{\xi}\Psi)^2}{2}\right)+\mu(\p^2_{\xi}\Psi)^2=\p_{\xi}\left\{\dots\right\}-(\sigma(\mathbf{v})-\sigma(\mathbf{\tilde{v}}))\p^2_{\xi}\Psi-\mathbf{H}_2\p^2_{\xi}\Psi,
		\end{equation}
		where $\left\{\dots\right\}=\frac{s(\p_{\xi}\Psi)^2}{2}+(\sigma(\mathbf{v})-\sigma(\mathbf{\tilde{v}}))\p_{\xi}\Psi+\mu\p^2_{\xi}\Psi\p_{\xi}\Psi+\mathbf{H}_2\p_{\xi}\Psi$ which vanishes at far fields $\xi=\pm\infty$. It is clear that
		\begin{equation}\label{high2.1}
			\begin{aligned}
				&\int_{0}^{\T}\int_{\R}\abs{(\sigma(\mathbf{v})-\sigma(\mathbf{\tilde{v}}))\p^2_{\xi}\Psi-\mathbf{H}_2\p^2_{\xi}\Psi}\,d\xi\,dt\\
				&\leq \frac{\mu}{4}\int_{0}^{\T}\int_{\R}\abs{\p^2_{\xi}\Psi}^2\,d\xi\,dt+C\int_{0}^{\T}\norm{\p_{\xi}\Phi}^2\,dt+C\varepsilon+C\varepsilon\int_{0}^{\T}\norm{\p^2_{\xi}\Psi}^2 \,dt,  
			\end{aligned}
		\end{equation}
		Then, by integrating \eqref{relt5} over $\R\times[0,\T]$, using Lemma \ref{lowterm1} and \eqref{high2.1}, and choosing $\varepsilon_0$ is small enough, we can get that
		\begin{equation}\label{estimate3}
			\supt\norm{\p_{\xi}\Psi}^2+\int_{0}^{\T}\norm{\p^2_{\xi}\Psi}^2\,dt\leq C(\norm{\Phi_0}^2_2+\norm{\Psi_0}^2_1+\varepsilon). 
		\end{equation}
		
		Third, differentiating \subeqref{system2}{2} with respect to $\xi$ twice and multiplying $\p^2_{\xi}\Psi$
		on the resulting equation, give that
		\begin{equation}\label{relt6}
			\p_t\left(\frac{(\p^2_{\xi}\Psi)^2}{2}\right)+\mu(\p^3_{\xi}\Psi)^2=\p_{\xi}\left\{\dots\right\}-[\sigma'(\mathbf{v})\p_{\xi}\Psi+(\sigma'(\mathbf{v})-\sigma'(\mathbf{\tilde{v}}))\p_{\xi}\mathbf{\tilde{v}}]\p^3_{\xi}\Psi
			-\p_{\xi}\mathbf{H}_2\p^3_{\xi}\Psi
		\end{equation}
		where $\left\{\dots\right\}=\frac{s(\p^2_{\xi}\Psi)^2}{2}+\p_{\xi}(\sigma(\mathbf{v})-\sigma(\mathbf{\tilde{v}}))\p^2_{\xi}\Psi+\mu\p^3_{\xi}\Psi\p^2_{\xi}\Psi+\p_{\xi}\mathbf{H}_2\p^2_{\xi}\Psi$. It follows from \eqref{v-bdd}, \eqref{vsx'decay}that
		\begin{equation}\label{high3.1}
			\begin{aligned}
				&\int_{0}^{\T}\int_{\R}\abs{[\sigma'(\mathbf{v})\p_{\xi}\Psi+(\sigma'(\mathbf{v})-\sigma'(\mathbf{\tilde{v}}))\p_{\xi}\mathbf{\tilde{v}}]\p^3_{\xi}\Psi}\,d\xi\,dt\\
				&\leq C\int_{0}^{\T}\int_{\R}\abs{\p_{\xi}\Phi\p^3_{\xi}\Psi+\p_{\xi}\Psi\p^3_{\xi}\Psi}\,d\xi\,dt\\
				&\leq\frac{\mu}{2}\int_{0}^{\T}\int_{\R}\abs{\p^3_{\xi}\Psi}^2\,d\xi\,dt+C\int_{0}^{\T}\int_{\R}\abs{\p_{\xi}\Phi}^2+\abs{\p_{\xi}\Psi}^2\,d\xi\,dt,
			\end{aligned}
		\end{equation}
		It follows from Lemma \ref{decay-error} that
		\begin{equation}\label{high3.2}
			\int_{0}^{\T}\int_{\R}\abs{\p_{\xi}\mathbf{H}_2\p^3_{\xi}\Psi}\,d\xi\,dt\leq\int_{0}^{\T}\norm{\p_{\xi}\mathbf{H}_2}\norm{\p^3_{\xi}\Psi}\,dt\leq C\varepsilon+C\varepsilon\int_{0}^{\T}\norm{\p^3_{\xi}\Psi}^2\,dt. 
		\end{equation}
		By intergrating \eqref{relt6} over $\R\times[0,\T]$, plugging \eqref{high3.1} and \eqref{high3.2}, and choosing $\varepsilon_0$ small enough, we obtain that
		\begin{equation}\label{estimate4}
			\supt\norm{\p^2_{\xi}\Psi}^2+\int_{0}^{\T}\norm{\p^3_{\xi}\Psi}^2\,dt\leq C(\norm{\Phi_0}^2_2+\norm{\Psi_0}^2_2+\varepsilon).
		\end{equation}
		
		Finally, differentiating \subeqref{system2}{1} with respect to $\xi$ twice and multiplying $\p^2_{\xi}\Phi$
		on the resulting equation imply that
		\begin{equation}\label{relt7}
			\p_t\left(\frac{(\p^2_{\xi}\Phi)^2}{2}\right)=\p^3_{\xi}\Psi\p^2_{\xi}\Phi+\p^2_{\xi}\mathbf{H}_1\p^2_{\xi}\Phi+\p_{\xi}\left(\frac{s(\p^2_{\xi}\Phi)^2}{2}\right),
		\end{equation}
		Thus, it follows from \eqref{estimate1}, \eqref{estimate2}, \eqref{estimate3} and \eqref{estimate4} that
		\begin{align*}
			\supt\norm{\p^2_{\xi}\Phi}^2 &\leq C\left(\norm{\Phi_0}^2_2+\int_{0}^{\T}\norm{\p_{\xi}\Phi}^2_1\,dt+\int_{0}^{\T}\norm{\p^3_{\xi}\Psi}^2\,dt+\int_{0}^{\T}\norm{\p^2_{\xi}\mathbf{H}_1}^2\,dt\right)\\
			&\leq C(\norm{\Phi_0}^2_2+\norm{\Psi}^2_2+\varepsilon).
		\end{align*}
	Combining the estimates above, and  Lemma \ref{lowterm1}, \ref{lowterm2}, we can comlete the proof of this Lemma.
\end{proof}
As a consequence of Lemma \ref{lowterm1}, \ref{lowterm2} and \ref{lowterm3}, one can immediately get Proposition \ref{mainpro}.
\section{Proof of Theorem \ref{mainresult}}\label{section4}
For any $\T>0$, let
\begin{align*}
	\mathcal{B}(0,\T)=&\left\{(\Phi,\Psi) \in C\Big(0,\T;H^2(\R) \times H^2(\R)\Big):\right.\\
	&\left.\p_{\xi}\Phi\in L^2\Big(0,\T;H^1(\R)\Big),\p_{\xi}\Psi \in L^2\Big(0,\T;H^2(\R)\Big)\right\}.
\end{align*}
The local existence and uniqueness of solution $(\Phi,\Psi)$ to \eqref{system2} and \eqref{newinitial} in  $\mathcal{B} (0,\T_0)$ can be proved by the standard contraction mapping theorem. Hence, with
Proposition \ref{mainpro}, one can let $\varepsilon_0>0$ be small such that $\mathop{sup}\limits_{t \in [0,\T_0]}\norm{\Phi,\Psi}^2_2 \leq C_0(\varepsilon^2_0+\varepsilon_0)<\delta^2_0$,
then the a priori assumptions \eqref{initialsmall} can be closed. Through a standard continuation
argument, one can obtain a global in time solution $(\Phi,\Psi) \in \mathcal{B}(0,+\infty)$ to\eqref{system2}, \eqref{newinitial}, and satisfy the estimate \eqref{proesti} with $\T=+\infty$, provided that the initial perturbations are small. This yields the global existence and uniqueness of solution to \eqref{VNC-Ion} and \eqref{initial} satisfying \eqref{regularity}. Moreover, with the estimate \eqref{proesti} holds for $\T=+\infty$ at hand, one can follow the same line as
in \cite{zero-mass} to prove that 	
\begin{equation}\label{far1}
	\norm{(v-\tilde{v},u-\tilde{u})}_{L^{\infty}(\R)}(t) \to 0\quad \text{as}~t \to +\infty.
\end{equation} 
The details
are omitted. 
%Thus, the proof of Theorem \ref{3.1} is completed.
%\begin{theorem}\label{3.1}
%	Under the assumptions of Theorem \ref{mainresult}, there exists $\varepsilon_0>0$, such that if $\varepsilon=E_0\leq \varepsilon_0$, then there is a unique solution
%	$(\Phi,\Psi) \in \mathcal{B} (0,+\infty)$ of \eqref{system2} and \eqref{newinitial}, satisfying
%	\begin{equation}\label{3.1est}
%		\mathop{sup}\limits_{t>0}\norm{\Phi,\Psi}^2_2(t)+\int_{0}^{+\infty}(\norm{\p_{\xi}\Phi}^2_1+\norm{\p_{\xi}\Psi}^2_2)(t)  \,dt\leq C\varepsilon.
%	\end{equation}
%	Moreover, it holds that
%	\begin{equation}\label{far1}
%		\norm{\phi,\psi}_{L^{\infty}(\R)}(t) \to 0\quad \text{as}~t \to +\infty.
%	\end{equation}
%\end{theorem}

%To complete the proof of Theorem \ref{mainresult}, it suffices to show \eqref{asymptotic}. In fact, since 
%\begin{align*}
%	\mathop{sup}\limits_{x \in \R}|(v-\vt)(x,t)|=\mathop{sup}\limits_{\xi \in \R}|(\mathbf{v}-\mathbf{\tilde{v}})(\xi,t)|=\norm{\phi}_{L^{\infty}} \to 0\quad \text{as}~t \to +\infty,
%\end{align*}
%and
Note that
\begin{equation}\label{asym1}
	\begin{aligned}
		&\mathop{sup}\limits_{x \in \R}|\vt(x,t)-v^S(x-st-\X_{\infty})|\leq \mathop{sup}\limits_{x \in \R}|\big(v^S(x-st-\X(t))-v^S(x-st-\X_{\infty})\big)|\\
		&+\mathop{sup}\limits_{x \in \R}|(v_l-\avl)(1-g_{\X})+(v_r-\avr)g_{\X}|\\
		&\leq C|\X(t)-\X_{\infty}|+\mathop{sup}\limits_{x \in \R}|v_l-\avl|+\mathop{sup}\limits_{x \in \R}|v_r-\avr|\leq Ce^{-\alpha t},
	\end{aligned}
\end{equation}
From \eqref{far1} and \eqref{asym1}, \eqref{asymptotic} holds for $v$. Similarly, the asymptotic of $u$ in \eqref{asymptotic} also holds true. Therefore, Theorem \ref{mainresult} is proved.

\appendix

\section{Some tedious calculations for the integral of $\mathbf{I_{1,6}}$}
\subsection{Proof of \eqref{interact}}

For 
\begin{align*}
	[\zeta](v_k):=\mu aY(\vsx)' \ph^2(v_k^+)-\mu aY\ph^2(\vsx)' (v_k^-)
\end{align*}
where $k=1,2,\dots, N-1$, we calculate 
the formula $k=1$ as an example, others are similar. When $\vsx \in [v_0,v_1]$, $a(\vsx)=[H(\vsx+\frac{1}{2H})]^{\frac{5}{4}}$, $T(\vsx)=\frac{a'(\vsx)}{4a(\vsx)}$, when $\vsx \in [v_1,v_2]$, $a(\vsx)=[H(\frac{1}{H}+v_1-\vsx)]^{\frac{5}{4}}$, 
$T(\vsx)=\frac{3a'(\vsx)}{4a(\vsx)}$, we have
\begin{equation}\label{apd1}
	\begin{aligned}
		&[\zeta](v_1)=\mu aY\ph^2(\vsx)'(v_1^+)-\mu aY\ph^2(\vsx)'(v_1^-)\\
		&=\frac{\mu}{4}(\vsx)'(v_1)\ph^2(v_1)[-a'(\vsx)(v_1^+)-a'(\vsx)(v_1^-)]=\frac{5\mu}{16} (\vsx)'(v_1)\ph^2(v_1)\cdot (1-1)=0.
	\end{aligned}
\end{equation}
then we calculate 
\begin{align*}
	[\zeta](v_N)=\mu aY(\vsx)'\ph^2(v_N^+)-\mu aY(\vsx)'\ph^2(v_N^-),
\end{align*}
where $v_N=v_b$. When $\vsx \in [v_{N-1},v_N]$, $a(\vsx)=[H(\frac{1}{H}+v_{N-1}-\vsx)]^{\frac{5}{4}}$, $T(\vsx)=\frac{3a'(\vsx)}{4a(\vsx)}$, $\vsx \in [v_N,\avr]$, $a(\vsx)=\frac{C_0\sigma'(\vsx)}{(s^2-\sigma'(\vsx))^2}$, $T(\vsx)=C_N \cdot e^{\frac{3}{4}ln(a(v_b))}$,
by \eqref{Hbig}, we have 
\begin{equation}\label{apd2}
	\begin{aligned}
		[\zeta](v_N)=\frac{\mu}{2}\left(-\frac{5H}{2^{\frac{13}{4}}}+\frac{(s^2+\sigma'(v_b))\sigma''(v_b)}{2^{\frac{5}{4}}(s^2-\sigma'(v_b))}\right)\leq 0.
	\end{aligned}
\end{equation}

\subsection{Proof of \eqref{interact2}}
It follows from the formula of $\wh,~\th$ in\eqref{wh} and \eqref{th} that, for $\vsx\in[0,v_d]$, 
$$Y(\vsx)=-\frac{1}{2}\frac{\hat{a}'}{a}(\vsx),\quad a'(\vsx)=\sigma''(\vsx)\wh(\vsx)+\sigma'(\vsx)\left(-\frac{v_d^2-(\vsx)^2}{\sigma''(\vsx)}+\frac{-2\vsx}{\sigma'(\vsx)}\right)$$ which yields that
$$\hat{a}(v_d)=\wh(v_d)=0, ~~-\mu aY(v_d)=\frac{\mu}{2}\hat{a'}=-\frac{\mu}{2}v_d$$ 
Thus, from \eqref{zeta},  we have
\begin{align*}
	\zeta(\xi_d+\X(t))=-\big[\frac{\mu v_d}{2}(\vsx)'\ph^2\big](\xi_d+\X(t))\leq 0.
\end{align*}

\

\noindent {\bf Acknowledgments:}
The research of Yu Mei is supported by the National Natural Science Foundation of China No. 12101496 and 12371227.	

\bigskip

\noindent {\bf Data availability }Data sharing not applicable to this article as no datasets were generated or analyzed during
the current study.

\section*{Declarations}

\noindent {\bf Conflict of interest} The authors declared that they have no conflict of interest to this work.

\bibliographystyle{plain}
\bibliography{visco-ref.bib}% common bib file

\begin{thebibliography}{10}

\bibitem{Zumbrun2016}
Blake Barker and Kevin Zumbrun.
\newblock Numerical proof of stability of viscous shock profiles.
\newblock {\em Math. Models Methods Appl. Sci.}, 26(13):2451--2469, 2016.

\bibitem{Lax2}
James Glimm and Peter~D. Lax.
\newblock {\em Decay of solutions of systems of nonlinear hyperbolic
  conservation laws}, volume No. 101 of {\em Memoirs of the American
  Mathematical Society}.
\newblock American Mathematical Society, Providence, RI, 1970.

\bibitem{Goodman.1986}
Jonathan Goodman.
\newblock Nonlinear asymptotic stability of viscous shock profiles for
  conservation laws.
\newblock {\em Arch. Rational Mech. Anal.}, 95(4):325--344, 1986.

\bibitem{GoodmanTAMS}
Jonathan Goodman.
\newblock Stability of viscous scalar shock fronts in several dimensions.
\newblock {\em Trans. Amer. Math. Soc.}, 311(2):683--695, 1989.

\bibitem{HehuangJDE}
Lin He and Feimin Huang.
\newblock Nonlinear stability of large amplitude viscous shock wave for general
  viscous gas.
\newblock {\em J. Differential Equations}, 269(2):1226--1242, 2020.

\bibitem{HuangWangZhang}
Feimin Huang, Yi~Wang, and Jian Zhang.
\newblock Time-asymptotic stability of composite waves of degenerate {O}leinik
  shock and rarefaction for non-convex conservation laws.
\newblock {\em Math. Ann.}, 392(1):1--46, 2025.

\bibitem{HuangYuanCMP}
Feimin Huang and Qian Yuan.
\newblock Stability of large-amplitude viscous shock under periodic
  perturbation for 1-d isentropic {N}avier-{S}tokes equations.
\newblock {\em Comm. Math. Phys.}, 387(3):1655--1679, 2021.

\bibitem{Zumbrun2009}
Jeffrey Humpherys, Gregory Lyng, and Kevin Zumbrun.
\newblock Spectral stability of ideal-gas shock layers.
\newblock {\em Arch. Ration. Mech. Anal.}, 194(3):1029--1079, 2009.

\bibitem{Zumbrun2017}
Jeffrey Humpherys, Gregory Lyng, and Kevin Zumbrun.
\newblock Multidimensional stability of large-amplitude {N}avier-{S}tokes
  shocks.
\newblock {\em Arch. Ration. Mech. Anal.}, 226(3):923--973, 2017.

\bibitem{Oleinik.1960}
A.~M. Il'in and O.~A. Oleinik.
\newblock Asymptotic behavior of solutions of the {C}auchy problem for some
  quasi-linear equations for large values of the time.
\newblock {\em Mat. Sb. (N.S.)}, 51(93):191--216, 1960.

\bibitem{vasseurjmes}
Moon-Jin Kang and Alexis Vasseur.
\newblock Contraction property for large perturbations of shocks of the
  barotropic {N}avier-{S}tokes system.
\newblock {\em J. Eur. Math. Soc. (JEMS)}, 23(2):585--638, 2021.

\bibitem{vasseur2}
Moon-Jin Kang and Alexis~F. Vasseur.
\newblock {$L^2$}-contraction for shock waves of scalar viscous conservation
  laws.
\newblock {\em Ann. Inst. H. Poincar\'e{} C Anal. Non Lin\'eaire},
  34(1):139--156, 2017.

\bibitem{vasseurinvent}
Moon-Jin Kang and Alexis~F. Vasseur.
\newblock Uniqueness and stability of entropy shocks to the isentropic {E}uler
  system in a class of inviscid limits from a large family of {N}avier-{S}tokes
  systems.
\newblock {\em Invent. Math.}, 224(1):55--146, 2021.

\bibitem{vasseurwang}
Moon-Jin Kang, Alexis~F. Vasseur, and Yi~Wang.
\newblock Time-asymptotic stability of composite waves of viscous shock and
  rarefaction for barotropic {N}avier-{S}tokes equations.
\newblock {\em Adv. Math.}, 419:Paper No. 108963, 66, 2023.

\bibitem{kawashima1993}
Shuichi Kawashima.
\newblock Global solutions to the equation of viscoelasticity with fading
  memory.
\newblock {\em J. Differential Equations}, 101(2):388--420, 1993.

\bibitem{MatsumuraCMP}
Shuichi Kawashima and Akitaka Matsumura.
\newblock Asymptotic stability of traveling wave solutions of systems for
  one-dimensional gas motion.
\newblock {\em Comm. Math. Phys.}, 101(1):97--127, 1985.

\bibitem{InitialModel}
Shuichi Kawashima and Akitaka Matsumura.
\newblock Stability of shock profiles in viscoelasticity with non-convex
  constitutive relations.
\newblock {\em Comm. Pure Appl. Math.}, 47(12):1547--1569, 1994.

\bibitem{Lax1}
P.~D. Lax.
\newblock Hyperbolic systems of conservation laws. {II}.
\newblock {\em Comm. Pure Appl. Math.}, 10:537--566, 1957.

\bibitem{MeiYuan2025}
Yeping Li, Yu~Mei, and Yuan Yuan.
\newblock Asymptotic stability of shock profiles and rarefaction waves to the
  {N}avier-{S}tokes-{P}oisson system under space-periodic perturbations.
\newblock {\em J. Differential Equations}, 420:400--449, 2025.

\bibitem{LiuGuo}
Jinjing Liu and Zhenhua Guo.
\newblock Nonlinear stability of large amplitude viscous shock waves to the
  one-dimensional system of viscoelasticity.
\newblock {\em SIAM J. Math. Anal.}, 53(2):1818--1830, 2021.

\bibitem{Liu1}
Tai-Ping Liu.
\newblock Nonlinear stability of shock waves for viscous conservation laws.
\newblock {\em Mem. Amer. Math. Soc.}, 56(328):v+108, 1985.

\bibitem{Liu1988}
Tai-Ping Liu.
\newblock Nonlinear waves for viscoelasticity with fading memory.
\newblock {\em J. Differential Equations}, 76(1):26--46, 1988.

\bibitem{Liu2}
Tai-Ping Liu and Yanni Zeng.
\newblock Shock waves in conservation laws with physical viscosity.
\newblock {\em Mem. Amer. Math. Soc.}, 234(1105):vi+168, 2015.

\bibitem{Mascia2004}
Corrado Mascia and Kevin Zumbrun.
\newblock Stability of large-amplitude viscous shock profiles of
  hyperbolic-parabolic systems.
\newblock {\em Arch. Ration. Mech. Anal.}, 172(1):93--131, 2004.

\bibitem{Gaijin}
Akitaka Matsumura and Ming Mei.
\newblock Nonlinear stability of viscous shock profile for a non-convex system
  of viscoelasticity.
\newblock {\em Osaka J. Math.}, 34(3):589--603, 1997.

\bibitem{zero-mass}
Akitaka Matsumura and Kenji Nishihara.
\newblock On the stability of travelling wave solutions of a one-dimensional
  model system for compressible viscous gas.
\newblock {\em Japan J. Appl. Math.}, 2(1):17--25, 1985.

\bibitem{Matsumura-Nishihara1995}
Akitaka Matsumura and Kenji Nishihara.
\newblock Asymptotic stability of traveling waves for scalar viscous
  conservation laws with non-convex nonlinearity.
\newblock {\em Comm. Math. Phys.}, 165(1):83--96, 1994.

\bibitem{meiM3AS}
Ming Mei.
\newblock Stability of shock profiles for nonconvex scalar viscous conservation
  laws.
\newblock {\em Math. Models Methods Appl. Sci.}, 5(3):279--296, 1995.

\bibitem{Meiming.1997}
Ming Mei and Kenji Nishihara.
\newblock Nonlinear stability of travelling waves for one-dimensional
  viscoelastic materials with non-convex nonlinearity.
\newblock {\em Tokyo J. Math.}, 20(1):241--264, 1997.

\bibitem{Nishihara.1995}
Kenji Nishihara.
\newblock Stability of traveling waves with degenerate shock for system of
  one-dimensional viscoelastic model.
\newblock {\em J. Differential Equations}, 120(2):304--318, 1995.

\bibitem{Xin1}
Anders Szepessy and Zhou~Ping Xin.
\newblock Nonlinear stability of viscous shock waves.
\newblock {\em Arch. Rational Mech. Anal.}, 122(1):53--103, 1993.

\bibitem{Yao.2016}
Alexis~F. Vasseur and Lei Yao.
\newblock Nonlinear stability of viscous shock wave to one-dimensional
  compressible isentropic {N}avier-{S}tokes equations with density dependent
  viscous coefficient.
\newblock {\em Commun. Math. Sci.}, 14(8):2215--2228, 2016.

\bibitem{wangjems}
Teng Wang and Yi~Wang.
\newblock Nonlinear stability of planar viscous shock wave to three-dimensional
  compressible {{Navier}}--{{Stokes}} equations.
\newblock {\em J. Eur. Math. Soc.}, 2024.

\bibitem{Xin.2019}
Zhouping Xin, Qian Yuan, and Yuan Yuan.
\newblock Asymptotic stability of shock waves and rarefaction waves under
  periodic perturbations for 1-{D} convex scalar conservation laws.
\newblock {\em SIAM J. Math. Anal.}, 51(4):2971--2994, 2019.

\bibitem{Xin.2021}
Zhouping Xin, Qian Yuan, and Yuan Yuan.
\newblock Asymptotic stability of shock profiles and rarefaction waves under
  periodic perturbations for 1-{D} convex scalar viscous conservation laws.
\newblock {\em Indiana Univ. Math. J.}, 70(6):2295--2349, 2021.

\bibitem{YuanSiam}
Qian Yuan.
\newblock Planar viscous shocks with periodic perturbations for scalar
  multidimensional viscous conservation laws.
\newblock {\em SIAM J. Math. Anal.}, 55(3):1499--1523, 2023.

\bibitem{YuanSCM2025}
Qian Yuan.
\newblock Time-asymptotic stability of planar {N}avier-{S}tokes shocks with
  spatial oscillations.
\newblock {\em Sci. China Math.}, 68(6):1355--1392, 2025.

\bibitem{Yuan.2020}
Qian Yuan and Yuan Yuan.
\newblock On {R}iemann solutions under different initial periodic perturbations
  at two infinities for 1-d scalar convex conservation laws.
\newblock {\em J. Differential Equations}, 268(9):5140--5155, 2020.

\bibitem{Zumbrun1998}
Kevin Zumbrun and Peter Howard.
\newblock Pointwise semigroup methods and stability of viscous shock waves.
\newblock {\em Indiana Univ. Math. J.}, 47(3):741--871, 1998.

\end{thebibliography}
	
%\begin{thebibliography}{99}
%
%
%
%\bibitem{M}
%Author names.
%\emph{paper title}.
%\newblock Journal title, \textbf{no.}(Issue), page, year.
%
%
%		
%		
%\end{thebibliography}

\end{document}